\documentclass[11pt, twoside]{article}
\usepackage{amsfonts}

\usepackage{mathrsfs}
\usepackage{cite}
\usepackage{amsmath,amssymb}
\usepackage{multicol}
\usepackage{graphicx}
\usepackage{epsfig}
\usepackage{color}
\usepackage{amsthm}
\usepackage{indentfirst}
\usepackage{graphics,amsmath,amssymb,amsthm,mathrsfs,bm}

\numberwithin{equation}{section}
\newtheorem{theorem}{Theorem}[section]
\newtheorem{lemma}[theorem]{Lemma}

\newtheorem{definition}[theorem]{Definition}
\newtheorem{remark}[theorem]{Remark}
\newtheorem{proposition}[theorem]{Proposition}
\allowdisplaybreaks

\makeatletter
\begin{document}

\author{Li-Jun Du, Wan-Tong Li\thanks{%
Corresponding author (wtli@lzu.edu.cn).}~  and Jia-Bing Wang\\
School of Mathematics and Statistics, Lanzhou University,\\
Lanzhou, Gansu 730000, People's Republic of China}

\title{\textbf{Asymptotic Behavior of Traveling Fronts and Entire Solutions for a Periodic Bistable Competition-Diffusion System}}
\date{}

\maketitle

\begin{abstract}
This paper is concerned with a time periodic competition-diffusion system
\begin{equation*}
\begin{cases}
{u_t}={u_{xx}}+u(r_1(t)-a_1(t)u-b_1(t)v),\quad t>0,~x\in \mathbb R,\\
{v_t}=d{v_{xx}}+v(r_2(t)-a_2(t)u-b_2(t)v),\quad t>0,~x\in \mathbb R,
\end{cases}
\end{equation*}
where $u(t,x)$ and $v(t,x)$ denote the densities of two competing species,
$d>0$ is some constant, $r_i(t),a_i(t)$ and $b_i(t)$ are $T-$periodic continuous functions.
Under suitable conditions, it has been confirmed by Bao and Wang [\emph{J. Differential Equations} 255 (2013), 2402-2435] that this system admits a periodic traveling front connecting two \textbf{stable} semi-trivial $T-$periodic solutions $(p(t),0)$ and $(0,q(t))$ associated to the corresponding kinetic system. Assume further that the wave speed is non-zero, we investigate the asymptotic behavior of the periodic \textbf{bistable} traveling front at infinity by a dynamical approach combined with the two-sided Laplace transform method. With these asymptotic properties, we then give some key estimates. Finally, by applying super- and subsolutions technique as well as the comparison principle, we establish the existence and various qualitative properties of \emph{entire solutions} defined for all time and whole space.


\end{abstract}

\newpage


\section{Introduction and main results}

In this paper, we study the asymptotic behavior of traveling wave fronts and
entire solutions for the following time periodic Lotka-Volterra competition-diffusion system
\begin{equation}\label{uv-eq}
\begin{cases}
{u_t}={u_{xx}}+u(r_1(t)-a_1(t)u-b_1(t)v),\quad t>0,~x\in \mathbb R,\\
{v_t}=d{v_{xx}}+v(r_2(t)-a_2(t)u-b_2(t)v),\quad t>0,~x\in \mathbb R,
\end{cases}
\end{equation}
where $u=u(t,x)$ and $v=v(t,x)$ denote the densities of two competing species
at time $t>0$ and in location $x\in \mathbb R,$
$d>0$ is the relative diffusive coefficient of the two species,
$r_i$, $a_i$ and $b_i$ are $T-$periodic continuous functions of $t$, $a_i$ and $b_i$ are positive in $[0,T]$, and $\overline{r_i}:=\frac{1}{T}\int_0^T{r_i(t)dt}>0$ with $i=1,2$. Usually, system \eqref{uv-eq} is used to describe the evolution of two competing species which live in a
fluctuating environment, exactly, many physical environmental conditions such
as temperature, humidity, the availability of food, water and other
resources usually vary in time with seasonal or daily variations \cite{zhao2003}.
In particular, we plan to investigate system \eqref{uv-eq} in the periodic framework, which is probably the simplest but nonetheless interesting and realistic case.

Time periodic traveling wave solutions of \eqref{uv-eq} connecting
$(0,q(t))$ and $(p(t),0)$ are classical solutions with the form
$
(u(t,x),v(t,x))=(X(t,x-ct),Y(t,x-ct))
$
satisfying
\begin{equation*}
\begin{aligned}
(X(t+T,z),Y(t+T,z))&=(X(t,z),Y(t,z)),\quad(t,x)\in\mathbb R\times\mathbb R,\\
\lim\limits_{z\to-\infty}(X(t,z),Y(t,z))&=(0,q(t)) \text{ uniformly in } t\in\mathbb R,\\
\lim\limits_{z\to+\infty}(X(t,z),Y(t,z))&=(p(t),0) \text{ uniformly in } t\in\mathbb R,
\end{aligned}
\end{equation*}
where $c\in \mathbb{R}$ is the wave speed, $z=x-ct$ is the co-moving frame coordinate,
$(0,q(t))$ and $(p(t),0)$ are $T-$periodic
solutions of the corresponding kinetic system
\begin{equation}\label{uv-oeq}
\begin{cases}
\frac{du}{dt}=u(r_1(t)-a_1(t)u-b_1(t)v),\\
\frac{dv}{dt}=v(r_2(t)-a_2(t)u-b_2(t)v),
\end{cases}
\end{equation}
which are explicitly given by
\begin{equation}\label{p-q(t)}
\begin{cases}
p(t)=\frac{p_0e^{\int_0^tr_1(s)ds}}{1+p_0\int_0^te^{\int_0^sr_1(\tau)d\tau}a_1(s)ds},
~p_0=\frac{e^{\int_0^Tr_1(s)ds}-1}{\int_0^Te^{\int_0^sr_1(\tau)d\tau}a_1(s)ds},\\
q(t)=\frac{q_0e^{\int_0^tr_2(s)ds}}{1+q_0\int_0^te^{\int_0^sr_2(\tau)d\tau}b_2(s)ds},
~q_0=\frac{e^{\int_0^Tr_2(s)ds}-1}{\int_0^Te^{\int_0^sr_2(\tau)d\tau}b_2(s)ds}.
\end{cases}
\end{equation}

For system \eqref{uv-eq} with autonomous nonlinearities,
the dynamical behaviors especially for traveling wave solutions
have been understood very well,
see, e.g., \cite{hosono1989,kan-on1995,tang1980,vuuren1995,guo2011,guo2012}.
Recently, there have been quite a few works focusing on
the nonautonomous Lotka-Volterra competition-diffusion system.
Among others, Zhao and Ruan \cite{zhao2011} established the existence,
uniqueness and stability of time periodic traveling wave fronts for system \eqref{uv-eq}
under \textbf{monostable} assumptions. Later, they extended these results to a class of periodic advection-reaction-diffusion systems in \cite{zhao2014}. With respect to the \textbf{bistable} situation, Bao and Wang \cite{bao2013} studied the existence, uniqueness and stability of time periodic traveling wave fronts, while Bao et al. \cite{bao2015} further considered those properties of time periodic traveling curved fronts in two dimensional space.

Throughout the paper, we always assume that
\begin{description}
\item[(A1)]
     $r_i, a_i, b_i\in C^\theta(\mathbb{R},\mathbb{R})$ with $0<\theta<1.$
     $\overline{r_i}>0, a_i(t)>0$ and $b_i(t)>0$ for any $t\in [0,T].$
     $r_i(t+T)=r_i(t), a_i(t+T)=a_i(t), b_i(t+T)=b_i(t), i=1,2.$
\item[(A2)]
    $\overline{r_1}
    <\mathop{\min}\limits_{t\in [0,T]}\left(\frac{b_1(t)}{b_2(t)}\right)\overline{r_2},
    ~\overline{r_2}
    <\mathop{\min}\limits_{t\in[0,T]}\left(\frac{a_2(t)}{a_1(t)}\right)\overline{r_1}.$
\item[(A3)] $\overline{r_1}+\overline{r_2}
    >\mathop{\max}\limits_{t\in[0,T]}\left(\frac{a_2(t)}{a_1(t)}\right)\overline{r_1}$,~ $\overline{r_1}+\overline{r_2}
    >\mathop{\max}\limits_{t\in[0,T]}\left(\frac{b_1(t)}{b_2(t)}\right)\overline{r_2}.$
\end{description}
Note that (A2) implies that the two semi-trivial
T-periodic solutions $(p(t),0)$ and $(0,q(t))$ are \textbf{stable} in the interior of the
positive quadrant $\mathbb R_+^2=\{(u,v)|~u>0,~v>0\}.$
(A3) might be a technique assumption which ensures that the eigenvalue
problem related to the linearized system of \eqref{uv-oeq} at $(p(t),0)$ and $(0,q(t))$
exactly admits a positive eigenvalue and the corresponding eigenfunction is positive,
which is necessary in establishing the existence of
periodic traveling wave fronts in \cite{bao2013}. Moreover, combining (A2) and \eqref{p-q(t)}, it follows that
\begin{equation*}
\begin{aligned}
\overline{r_1}=\frac{1}{T}\int_0^Ta_1(s)p(s)ds
<\mathop{\min}\limits_{t\in[0,T]}\left(\frac{b_1(t)}{b_2(t)}\right)\overline{r_2}
=\mathop{\min}\limits_{t\in [0,T]}\left(\frac{b_1(t)}{b_2(t)}\right)\frac{1}{T}\int_0^Tb_2(s)q(s)ds,\\
\overline{r_2}=\frac{1}{T}\int_0^Tb_2(s)q(s)ds
<\mathop{\min}\limits_{t\in[0,T]}\left(\frac{a_2(t)}{a_1(t)}\right)\overline{r_1}
=\mathop{\min}\limits_{t\in[0,T]}\left(\frac{a_2(t)}{a_1(t)}\right)\frac{1}{T}\int_0^Ta_1(s)p(s)ds,\\
\end{aligned}
\end{equation*}
which indicates that $\overline{b_1q-a_1p}>0$ and $\overline{a_2p-b_2q}>0.$

Let
$$
u^*(t,x)=\frac{u(t,x)}{p(t)},\quad v^*(t,x)=\frac{q(t)-v(t,x)}{q(t)},
$$
then \eqref{uv-eq} becomes~(omitting $*$ for simplicity)
\begin{equation}\label{uv-trans-eq}
\begin{cases}
u_t=u_{xx}+a_1pu\left[1-u-N_1(t)(1-v)\right],\\
v_t=dv_{xx}+b_2q(1-v)\left[N_2(t)u-v\right],
\end{cases}
\end{equation}
where
$$N_1(t)=\frac{b_1(t)q(t)}{a_1(t)p(t)} \text{ and }
N_2(t)=\frac{a_2(t)p(t)}{b_2(t)q(t)} \text{ for any } t\in\mathbb R,
$$
and the corresponding traveling wave system is
\begin{equation}\label{PQ}
\begin{cases}
P_t=P_{zz}+cP_z+a_1pP[1-P-N_1(t)(1-Q)],\\
Q_t=dQ_{zz}+cQ_z+b_2q(1-Q)[N_2(t)P-Q],\\
(P(t,z),Q(t,z))=(P(t+T,z),Q(t+T,z)),\\
\mathop{\lim }\limits_{z\to-\infty}(P,Q)=(0,0),
\mathop{\lim }\limits_{z\to+\infty}(P,Q)=(1,1).
\end{cases}
\end{equation}

We first state the existence result of periodic traveling wave fronts of \eqref{uv-trans-eq}.

\begin{proposition}{\rm\cite[Theorem 2.5]{bao2013}}
Assume (A1)-(A3). Then there exists $c\in\mathbb R$ such that
\eqref{PQ} admits a solution $(P(t,x-ct),Q(t,x-ct))$
satisfying $(P_z(t,z),Q_z(t,z))>(0,0)$ for any $(t,z)\in\mathbb R^+\times\mathbb R.$
\end{proposition}

\begin{remark}\rm
Observe that, though the existence of periodic traveling wave fronts has been
established in \cite{bao2013}, the sign of the wave speed $c$ remains an open problem.
In fact, it is not easy to determine the sign of $c$,
which is important to decide which species becomes dominant and eventually
occupies the whole domain. Therefore, it remains an interesting problem
to study the sign of wave speed of the bistable traveling wave fronts.
Particularly, when $c=0$, the propagation shall be failure and there
occurs standing waves of \eqref{uv-trans-eq}, which makes it very difficult to construct sub- and supersolutions for \eqref{uv-trans-eq}.
In the rest of this paper, we always assume that $c\neq0$.
\end{remark}

It is well known that the asymptotic behavior of the
traveling wave solution is of great importance in investigating
further properties of the traveling wave solution such as the uniqueness and
stability (see, e.g., \cite{volpert2000,likun2012,likun2016}),
since this often determines the choice of perturbation space.
On the other hand, the asymptotic behavior of traveling waves
is also crucial in constructing appropriate sub- and supersolutions,
which enables us to establish some other new types of entire solutions (see, e.g., \cite{morita2009,wu2010,lv2010,LZZ2015}).
In other words, it is very necessary and important to study the asymptotic behavior
of traveling wave fronts near the limiting states.
In the homogeneous case,
asymptotic behavior of the traveling wave solution for reaction-diffusion
equations can usually be obtained by the standard asymptotic theory
( see, e.g., \cite{Leung2008,morita2009,lv2010}) or by using various versions of the Ikehara's Theorem (see, e.g., \cite{carr2004,chen2006,guo2012,zhang2012}).
In the nonhomogeneous case, this subject becomes more complex and difficult. In particular, by establishing exponential lower and upper bounds and using the comparison argument,
Hamel \cite{hamel2008} obtained the exponential behavior of the traveling wave front for a reaction-advection-diffusion equation with a general \textbf{monostable} nonlinearity
in periodic excitable media as it approaches the \textbf{unstable} limiting state, while
Zhao \cite{zhao2009} considered a scalar \textbf{bistable} reaction-diffusion equation in
infinite cylinders and obtained the exact exponential decay rates of time periodic traveling wave fronts near the \textbf{stable} limiting states by applying the partial Fourier transform method.

It should be mentioned that, for the periodic \textbf{monostable} traveling wave solution $(P(t,z),Q(t,z))$ of system \eqref{uv-trans-eq},
Zhao and Ruan \cite{zhao2011} established the exact exponential decay rates as it approaches its \textbf{unstable} limiting state.
Very recently, the authors of the current paper \cite{du2016} further obtained
the exact exponential decay rates as it approaches
its \textbf{stable} limiting state under a technical condition
which ensures that the eigenvalue related to the linearized $u$-equation
is smaller than that related to the linearized $v$-equation,
which leads to the same decay rate for $P(t,z)$ and $Q(t,z)$ as $z\to+\infty.$
In the present paper, we continue to study the exact exponential decay rates
of periodic \textbf{bistable} traveling fronts of system \eqref{uv-eq}
as they approach their \textbf{stable} limiting states for the other two cases,
that is, the eigenvalue related to the linearized $u$-equation
is equal or greater than that related to the linearized $v$-equation.
Compared with the autonomous Lotka-Volterra competition system, the time dependence
of the coefficients causes substantial technical difficulties,
and one cannot use the standard asymptotic theory or the Ikehara's theorem to
study the exponential decay rates of a periodic traveling wave front,
and hence different techniques have to be
utilized to address this issue.

Although the traveling wave solution is of great significance in characterizing the dynamics of reaction-diffusion equations, there might be other interesting patterns.
More precisely, a new type of entire solution which behaves as a combination of traveling wave fronts as $t\to-\infty$ has been observed in various reaction-diffusion problems,
we refer to the earlier and original work of
Hamel and Nadirashvili \cite{hamel1999,hamel2001} and Yagisita \cite{yagisita2003},
see also \cite{fukao2004,chen2005,guo2005,morita2006,li2008,wang2009} for equations with and without delays, and \cite{li2010,sun2011} with nonlocal dispersal.
In regard to systems, Morita and Tachibana \cite{morita2009} first established the existence of entire solutions for a homogeneous Lotka-Volterra competition-diffusion system, and similar results were showed by Guo and Wu \cite{wu2010} for the discrete version and
Li et al. \cite{LZZ2015} for the nonlocal dispersal version.
Note that all these works mainly concerned with space/time homogeneous equations.
Recently, many researchers devoted to the study of entire solutions for space/time periodic equations, see, e.g., \cite{bu2016,LWZ2016,liu2012}.
In particular, Du et al. \cite{du2016} established the
invasion entire solutions for system \eqref{uv-eq} with \textbf{monostable} structure.
In this paper, we will study the existence and various properties
of entire solutions for system \eqref{uv-eq} with \textbf{bistable} structure.

In order to employ the basic idea developed in \cite{morita2009} to
establish the existence of entire solutions,
that is, constructing a pair of appropriate sub- and supersolutions
and using the comparison argument,
we need some estimates which are concerned with the asymptotic behavior of
periodic traveling wave fronts.
One of the main difficulties arises in obtaining the exact exponential
decay rate of the periodic traveling wave front
as it approaches its limiting states.
Inspired by \cite{zhao2011},
we apply the two-sided Laplace transform method
to obtain the exact exponential decay rate of the periodic
traveling front as it approaches its \textbf{stable} limiting states,
which is essentially based on the a priori exponential estimates of
periodic traveling wave tails at infinity.
Here we would like to emphasize that, unlike the a priori exponential estimates of
periodic traveling wave tails at the \textbf{unstable} limiting state                                                                                                             characterized by the principle eigenvalue associated with the linearized system
(see \cite[Lemma $3.3$]{zhao2011}),
the a priori exponential estimates at the \textbf{stable} limiting
state can only be characterized by
a perturbation of the corresponding principle eigenvalues
${\nu_{i,\epsilon}^\pm}~(i=2,3)$ with some small $\epsilon>0$ and $\nu_i^\pm~(i=1,4)$
(see Lemmas \ref{general-estimate-+} and \ref{general-estimate--}),
which causes additional difficulties and
makes it vastly more complicated to use the two-sided Laplace transform method.

Our first main results, about the exact exponential decay rates of the
periodic traveling wave front of system \eqref{uv-trans-eq}
as it tends to its limiting states, are stated as follows.

\begin{theorem}\label{+infinity}
Assume (A1)-(A3).
Let $(P(t,z),Q(t,z))$ be a traveling wave
solution of \eqref{uv-trans-eq} with $c\neq0,$ then
\begin{align*}
&\mathop{\lim}\limits_{z\to+\infty}\frac{1-Q(t,z)}{{k_1}{e^{\nu_2z}}\phi_2(t)}=1,
~\mathop{\lim}\limits_{z\to+\infty}\frac{{Q_z}(t,z)}{{k_1}{e^{\nu_2z}}\phi_2(t)}=-\nu_2
\text{ uniformly in } t\in \mathbb R,\\
&\mathop{\lim}\limits_{z\to+\infty}\frac{1-P(t,z)}{{k_1}{e^{\nu_2z}}\tilde{\phi}_1(t)}=1,
~\mathop{\lim}\limits_{z\to+\infty}\frac{{P_z}(t,z)}{{k_1}{e^{\nu_2z}}\tilde{\phi}_1(t)}
=-\nu_2\text{ uniformly in } t\in \mathbb R, \text{ if } \nu_1<\nu_2,\\
&\mathop{\lim}\limits_{z\to+\infty}\frac{1-P(t,z)}
{{\vartheta_1\left|{z}\right|}{e^{\nu_1z}}\phi_1(t)}=1,
~\mathop{\lim}\limits_{z\to+\infty}\frac{{P_z}(t,z)}
{{\vartheta_1\left|{z}\right|}{e^{\nu_1z}}\phi_1(t)}
=-\nu_1\text{ uniformly in } t\in \mathbb R, \text{ if } \nu_1=\nu_2,\\
&\mathop{\lim}\limits_{z\to+\infty}\frac{1-P(t,z)}{{k_2}{e^{\nu_1z}}\phi_1(t)}=1,
~\mathop{\lim}\limits_{z\to+\infty}\frac{{P_z}(t,z)}{{k_2}{e^{\nu_1z}}\phi_1(t)}=-\nu_1
 \text{ uniformly in } t\in\mathbb R, \text{ if } \nu_1>\nu_2,
\end{align*}
where
$k_i>0,~\nu_i<0~(i=1,2)$ and $\vartheta_1=\vartheta_1(k_1,\nu_1,c)>0 $
are some constants,
$\phi_i(t)~(i=1,2)$ and $\tilde{\phi}_1(t)$ are some positive $T-$periodic functions
in $\mathbb R.$
\end{theorem}

\begin{theorem}\label{-infinity}
Assume (A1)-(A3).
Let $(P(t,z),Q(t,z))$ be a traveling wave
solution of \eqref{uv-trans-eq} with $c\neq0,$ then
\begin{align*}
&\mathop{\lim}\limits_{z\to-\infty}\frac{P(t,z)}{{k_3}{e^{\nu_3z}}\psi_1(t)}=1,
~\mathop{\lim}\limits_{z\to-\infty}\frac{{P_z}(t,z)}{{k_3}{e^{\nu_3z}}\psi_1(t)}=\nu_3
\text{ uniformly in } t\in \mathbb R,\\
&\mathop{\lim}\limits_{z\to-\infty}\frac{Q(t,z)}{{k_3}{e^{\nu_3z}}\tilde{\psi}_2(t)}=1,
~\mathop{\lim}\limits_{z\to-\infty}\frac{{Q_z}(t,z)}{{k_3}{e^{\nu_3z}}\tilde{\psi}_2(t)}
=\nu_3\text{ uniformly in } t\in \mathbb R, \text{ if } \nu_4>\nu_3,\\
&\mathop{\lim}\limits_{z\to-\infty}\frac{Q(t,z)}{{\vartheta_2|z|}{e^{\nu_4z}}\psi_2(t)}=1,
~\mathop{\lim}\limits_{z\to-\infty}\frac{{Q_z}(t,z)}{{\vartheta_2|z|}{e^{\nu_4z}}\psi_2(t)}
=\nu_4\text{ uniformly in } t\in \mathbb R, \text{ if } \nu_4=\nu_3,\\
&\mathop{\lim}\limits_{z\to-\infty}\frac{Q(t,z)}{{k_4}{e^{\nu_4z}}\psi_2(t)}=1,
~\mathop{\lim}\limits_{z\to-\infty}\frac{{Q_z}(t,z)}{{k_4}{e^{\nu_4z}}\psi_2(t)}=\nu_4
~\text{ uniformly in } t\in \mathbb R, \text{ if } \nu_4<\nu_3,
\end{align*}
where $k_i>0,~\nu_i>0~(i=3,4)$ and
$\vartheta_2=\vartheta_2(k_3,\nu_4,c)>0$ are some constants,
$\psi_i(t)~(i=1,2)$ and $\tilde{\psi}_1(t)$ are some positive $T-$periodic functions
in $\mathbb R.$
\end{theorem}

With these asymptotic properties, we further construct a pair of appropriate sub- and supersolutions, and then establish the existence of entire solutions which behave as
two periodic traveling wave fronts propagating from both sides of $x$-axis.
In addition, for the traveling wave front $(P,Q)$ of \eqref{uv-trans-eq},
if the wave speed $c<0,$ we further assume that
\begin{description}
\item[(C1)] There exists a positive number $\eta_0$ such that $\frac{P(t,z)}{Q(t,z)}\geq \eta_0$ for any $(t,z)\in\mathbb R\times(-\infty,0],$
\end{description}
and if $c>0,$ we assume that
\begin{description}
\item[(C2)] There exists a positive number $\eta_1$ such that $\frac{1-Q(t,z)}{1-P(t,z)}\geq \eta_1$ for any $(t,z)\in\mathbb R\times[0,+\infty).$
\end{description}
A similar assumption to (C1) or (C2) has appeared in various papers to study entire solutions for Lotka-Volterra competition systems (see, e.g., \cite{morita2009,wu2010,LZZ2015}), which is technical but crucial in constructing sub- and supersolutions. Thankfully, it follows from Theorems \ref{-infinity} and \ref{+infinity} that (C1) and (C2) are valid provided that $\nu_3<\nu_4$ and $\nu_1<\nu_2$, respectively.

Hereafter, we denote
$\bm {u}=(u_1,u_2)\in C_b(\mathbb{R}^2,\mathbb{R}^2),$
the set of all bounded and uniformly continuous functions
from $\mathbb{R}^2$ into $\mathbb{R}^2$.
Moreover, we write $\bm{u}=\bm{v}$ if
$u_1=v_1$ and $u_2=v_2$ for any $(t,x)\in\mathbb R\times\mathbb R,$
$\bm{u}+\bm{v}=(u_1+v_1,u_2+v_2),$
$\left|\bm {u}-\bm {v}\right|=\left|u_1-v_1\right|+\left|u_2-v_2\right|$
and $\bm{u}=a$ if $u_1=a$ and $u_2=a$ for any constant $a.$
The other relations such as $\bm{u}<\bm{v},$ $\bm{u}\leq\bm{v},$
$\bm{u}<a,$ $\bm{u}\leq a,$ ``max'', ``min'', ``sup'' and ``inf''
are similarly to be understood componentwise.
Particularly, denote by $\bm{0}=(0,0)$ and $\bm{1}=(1,1).$

Our another main results, regarding the existence and
some qualitative properties of entire solutions of system \eqref{uv-trans-eq},
are summarized in the next two theorems.

\begin{theorem}\label{ES}
Assume (A1)-(A3).
Let $\bm\Phi(t,z)=(P(t,z),Q(t,z))$ be a traveling wave solution of \eqref{uv-trans-eq}
satisfying (C1) with $c<0.$
Then for any given constants $\theta_1,\theta_2\in\mathbb R,$
system \eqref{uv-trans-eq} admits an entire solution
$\bm{W}_{\theta_1,\theta_2}(t,x)
=(U_{\theta_1,\theta_2}(t,x),V_{\theta_1,\theta_2}(t,x))$
satisfying $\bm{0}<\bm{W}_{\theta_1,\theta_2}(t,x)<\bm{1}$
and
\begin{description}
\item[(i)]
$\bm{W}_{\theta_1,\theta_2}(t+T,x)=\bm{W}_{\theta_1,\theta_2}(t,x)$
or
$\bm{W}_{\theta_1,\theta_2}(t+T,x)>\bm{W}_{\theta_1,\theta_2}(t,x)$
for all $(t,x)\in\mathbb R\times\mathbb R.$
\item[(ii)]
\begin{equation}\label{UV-w-PQ}
\begin{aligned}
\lim_{t\to-\infty}
&\left\{\sup_{x\geq 0}
{\left|\bm{W}_{\theta_1,\theta_2}(t,x)-\bm\Phi(t,x-ct+\theta_1)\right|}\right .\\
&+\left .\sup_{x\leq0}
{\left|\bm{W}_{\theta_1,\theta_2}(t,x)-\bm\Phi(t,-x-ct+\theta_2)\right|}
\right\}
=0.
\end{aligned}
\end{equation}

\item[(iii)]
$\mathop{\lim}\limits_{k\to+\infty}\mathop{\sup}\limits_{(t,x)\in[0,T]\times\mathbb R}
\left|\bm{W}_{\theta_1,\theta_2}(t+kT,x)-\bm{1}\right|=0.$

\item[(iv)]
$\mathop{\lim}\limits_{k\to-\infty}\mathop{\sup}\limits_{(t,x)\in[0,T]\times[a,b]}
\left|\bm{W}_{\theta_1,\theta_2}(t+kT,x)\right|=0$ for any $a,b\in\mathbb{R}$ with $a<b$.

\item[(v)] $\mathop{\lim}\limits_{|x|\to+\infty}\mathop{\sup}\limits_{t\in[t_0,+\infty)}
\left|\bm{W}_{\theta_1,\theta_2}(t,x)-\bm{1}\right|=0$
for any $t_0\in\mathbb R.$

\item[(vi)]
In the sense of locally in $(t,x)\in\mathbb R\times\mathbb R,$
\begin{equation*}
\bm{W}_{\theta_1,\theta_2}(t,x)
\text{ converges to }
\left\{{\begin{array}{*{20}{c}}
\begin{aligned}
&\bm\Phi(t,x-ct+\theta_1) ~\text{ as } \theta_2\to-\infty,\\
&\bm\Phi(t,-x-ct+\theta_2) ~\text{ as } \theta_1\to-\infty, \\
&\bm{0} ~\text{ as } \theta_1\to-\infty \text{ and }\theta_2\to-\infty,\\
&\bm{1} ~\text{ as } \theta_1\to+\infty \text{ or }\theta_2\to+\infty.
\end{aligned}
\end{array}} \right.
\end{equation*}

\item[(vii)]
$\bm{W}_{\theta_1,\theta_2}(t,x)$ is monotone increasing w.r.t.
$\theta_1$ and $\theta_2$ for any $(t,x)\in \mathbb R\times\mathbb R.$

\item[(viii)]
$\bm{W}_{\theta,\theta}(t,x)=\bm{W}_{\theta,\theta}(t,-x)$
for any $\theta\in\mathbb R$ and $(t,x)\in\mathbb R\times\mathbb R.$

\item[(ix)]
For any $\theta_1^*,\theta_2^*\in\mathbb R$ satisfying
$\frac{\theta_1^*-\theta_1+\theta_2^*-\theta_2}{-2cT}\in\mathbb Z,$
there exists $(t_0,x_0)\in\mathbb R\times\mathbb R$ such that
$\bm{W}_{\theta_1^*,\theta_2^*}(\cdot,\cdot)=\bm{W}_{\theta_1,\theta_2}(\cdot+t_0,\cdot+x_0)$
on $\mathbb R\times\mathbb R.$

\end{description}
\end{theorem}

\begin{theorem}\label{ES-c>0}
Assume (A1)-(A3).
Let $\bm\Psi(t,z)=(P(t,z),Q(t,z))$ be a traveling wave solution of \eqref{uv-trans-eq}
satisfying (C2) with $c>0.$
Then for any given constants $\theta_1,\theta_2 \in \mathbb R,$
system \eqref{uv-trans-eq} admits an entire solution
$\bm{\widetilde{W}}_{\theta_1,\theta_2}(t,x)
=(\widetilde{U}_{\theta_1,\theta_2}(t,x),\widetilde{V}_{\theta_1,\theta_2}(t,x))$
satisfying $\bm{0}<\bm{\widetilde{W}}_{\theta_1,\theta_2}(t,x)<\bm{1}$
and
\begin{description}
\item[(i)]
$\bm{\widetilde{W}}_{\theta_1,\theta_2}(t+T,x)=\bm{\widetilde{W}}_{\theta_1,\theta_2}(t,x)$
or
$\bm{\widetilde{W}}_{\theta_1,\theta_2}(t+T,x)<\bm{\widetilde{W}}_{\theta_1,\theta_2}(t,x)$
for all $(t,x)\in\mathbb R\times\mathbb R.$

\item[(ii)]
\begin{equation*}
\begin{split}
\lim_{t\to-\infty}&\left\{\sup_{x\geq 0}
{\left|\bm{\widetilde{W}}_{\theta_1,\theta_2}(t,x)-\bm\Psi(t,-x-ct+\theta_1)\right|} \right .\\
&+\left .\sup_{x\leq 0}
{\left|\bm{\widetilde{W}}_{\theta_1,\theta_2}(t,x)-\bm\Psi(t,x-ct+\theta_2)\right|}
\right\}
=0.
\end{split}
\end{equation*}

\item[(iii)]
$\mathop{\lim}\limits_{k\to+\infty}\mathop{\sup}\limits_{(t,x)\in[0,T]\times\mathbb R}
\left|\bm{\widetilde{W}}_{\theta_1,\theta_2}(t+kT,x)\right|=0.$

\item[(iv)]
$\mathop{\lim}\limits_{k\to-\infty}\mathop{\sup}\limits_{(t,x)\in[0,T]\times[a,b]}
\left|\bm{\widetilde{W}}_{\theta_1,\theta_2}(t+kT,x)-\bm{1}\right|=0$ for any  $a,b\in\mathbb{R}$ with $a<b$.

\item[(v)] $\mathop{\lim}\limits_{|x|\to+\infty}\mathop{\sup}\limits_{t\in[t_0,+\infty)}
\left|\bm{\widetilde{W}}_{\theta_1,\theta_2}(t,x)\right|=0$
for any $t_0\in\mathbb R.$

\item[(vi)]
In the sense of locally in $(t,x)\in\mathbb R\times\mathbb R,$
\begin{equation*}
\bm{\widetilde{W}}_{\theta_1,\theta_2}(t,x)
\text{ converges to }
\left\{{\begin{array}{*{20}{c}}
\begin{aligned}
&\bm\Psi(t,-x-ct+\theta_1) ~\text{ as }\theta_2\to+\infty,\\
&\bm\Psi(t,x-ct+\theta_2) ~\text{ as } \theta_1\to+\infty, \\
&\bm{1} ~\text{ as }\theta_1\to+\infty \text{ and } \theta_2\to+\infty,\\
&\bm{0} ~\text{ as }\theta_1\to-\infty \text{ or } \theta_2\to-\infty.
\end{aligned}
\end{array}} \right.
\end{equation*}
\end{description}

Moreover,
the assertions \rm\textbf{(vii)}-\rm\textbf{(ix)} in Theorem \ref{ES-c>0} are valid.
\end{theorem}

\begin{remark}{\rm
Notice that the entire solutions obtained in Theorems \ref{ES} and \ref{ES-c>0} are ``\emph{annihilating-front}" type, which behave as two periodic \textbf{bistable} traveling  fronts approaching each other from both sides of the $x-$axis as $t\rightarrow-\infty$ and annihilating as time increases. It is worthy to mention that Morita and Ninomiya \cite{morita2006} have constructed another two types of ``\emph{merging-front}" entire solutions for some standard autonomous \textbf{bistable} reaction-diffusion equations. One behaves as two \textbf{monostable} fronts approaching each other from both sides of the $x-$axis and merging and converging
to a single \textbf{bistable} front while the other behaves as a \textbf{monostable} front merging with a \textbf{bistable} front and one chases another from the same side of $x-$axis.
Hence it is also interesting to explore these kinds of ``\emph{merging-front}" entire solutions for the nonautonomous \textbf{bistable} system \eqref{uv-trans-eq}, which left as our further consideration.}
\end{remark}

\begin{remark}{\rm
Recently, there are many results on \textbf{nonlocal dispersal} equations, we refer to \cite{amrt2010,bates, bls2016,KLS2010,li2010,LWZ2016,LZZ2015,sun2011,WR2015}. Naturally, it is interesting and meaningful to consider the nonlocal version of the time periodic Lotka-Volterra competition-diffusion system \eqref{uv-trans-eq}:
\begin{equation}\label{dlw}
\begin{cases}
{u_t}=J*u(x,t)-u(x,t)+u(r_1(t)-a_1(t)u-b_1(t)v),\quad t>0,~x\in \mathbb R,\\
{v_t}=d(J*v(x,t)-v(x,t))+v(r_2(t)-a_2(t)u-b_2(t)v),\quad t>0,~x\in \mathbb R,
\end{cases}
\end{equation}
where the nonlocal dispersal operator is defined by
$$
(\mathcal{D}u)(x,t)=(J*u)(x,t)-u(x,t)=\int_{\mathbb{R}}J(x-y)[u(y,t)-u(x,t)]dy.
$$
We leave it to the interested readers.
}
\end{remark}

The rest of the paper is organized as follows.
In Section 2, we investigate the exact exponential decay rates of
the periodic traveling wave front of \eqref{uv-trans-eq} as it approaches its limiting states.
Section 3 is devoted to a pair of sub- and supersolutions for constructing entire solutions.
In Section 4, we establish the existence and some qualitative properties
of entire solutions by a comparing argument.

\section{Asymptotic behavior of periodic traveling fronts}

In this section, we study the asymptotic behavior of
periodic traveling wave fronts near the limiting states.
We first consider the case of $z\to+\infty.$

Let $u^\star(t,x)=1-u(t,x)$ and $v^\star(t,x)=1-v(t,x),$ then
\eqref{uv-trans-eq} is transformed into the following cooperative system
(omitting $\star$ for simplicity)
\begin{equation}\label{uv-1-0}
\begin{cases}
u_t=u_{xx}+g(t,u,v),\\
v_t=dv_{xx}+h(t,u,v),
\end{cases}
\end{equation}
where
\begin{equation*}
\begin{cases}
g(t,u,v)=-(1-u)[a_1(t)p(t)u-b_1(t)q(t)v],\\
h(t,u,v)=-v[a_2(t)p(t)(1-u)-b_2(t)q(t)(1-v)].
\end{cases}
\end{equation*}
The corresponding traveling wave solution $(U(t,z),V(t,z))$ satisfies
\begin{equation}\label{UV}
\begin{cases}
U_t=U_{zz}+cU_z+g(t,U,V),\\
V_t=dV_{zz}+cV_z+h(t,U,V),\\
(U(t,z),V(t,z))=(U(t+T,z),V(t+T,z)),\\
\mathop{\lim}\limits_{z\to -\infty }(U,V)=(1,1),
\mathop{\lim}\limits_{z\to +\infty }(U,V)=(0,0).
\end{cases}
\end{equation}

Denote
$$\kappa_1=-\overline{g_u(t,0,0)}=\overline{a_1p},
~\nu_1=\frac{-c-\sqrt{c^2+4\kappa_1}}{2}<0,
~\phi_1(t)=e^{\int_0^t{g_u(s,0,0)ds}+\kappa_1t},
$$
and
$$\kappa_2=-\overline{h_v(t,0,0)}=\overline{a_2p-b_2q},
~\nu_2=\frac{-c-\sqrt{c^2+4d\kappa_2}}{2d}<0,
~\phi_2(t)=e^{\int_0^t{h_v(s,0,0)ds}+\kappa_2t}.$$
Due to $g(t,0,0)=h(t,0,0)=g(t,1,1)=h(t,1,1)=0,$ system \eqref{UV} can be written as
\begin{equation*}
\begin{cases}
U_t=U_{zz}+cU_z+U\int_0^1{g_u(t,\tau U,\tau V)}d\tau+V\int_0^1{g_v(t,\tau U,\tau V)}d\tau,\\
V_t=dV_{zz}+cV_z+U\int_0^1{h_u(t,\tau U,\tau V)}d\tau+V\int_0^1{h_v(t,\tau U,\tau V)}d\tau.
\end{cases}
\end{equation*}
Since $(U(\cdot,z),V(\cdot,z))$ is periodic in $t\in\mathbb R,$
and $(U,V)(t,z)$ is positive and bounded for any $(t,x)\in\mathbb R\times\mathbb R,$
the Harnack inequality for cooperative parabolic systems (see \cite{foldes2009,zhao2011,bao2013}) implies that there exists a positive constant $N$ such that
\begin{equation}\label{har-estimate}
(U(t,z),V(t,z))\leq N(U(t^\prime,z),V(t^\prime,z))
\text{ for any } z,~t,~t^\prime\in\mathbb{R}.
\end{equation}

\indent
We first give the a priori exponential estimates of periodic traveling wave fronts
of system \eqref{uv-1-0} as $z\to+\infty$ as follows.

\begin{lemma}\label{general-estimate-+}
Assume (A1)-(A3). Let $(U(t,z),V(t,z))$ be a solution of \eqref{UV}.
Then for any $\epsilon\in\left(0,\min\left\{1,\frac{\kappa_2}{C_1^+}\right\}\right),$
there exist positive constants $K_i,~K_i^\prime~(i=1,2)$ such that
\begin{equation}\label{general-estimate-UV+}
{K_1}{e^{\nu_1^-z}}\leq U(t,z)\leq K_1^\prime{e^{\nu_1^+z}},
\quad {K_2}{e^{\nu_{2,\epsilon}^-z}}\leq V(t,z)
\leq K_2^\prime{e^{\nu_{2,\epsilon}^+z}}
\end{equation}
for any $(t,z)\in\mathbb{R}\times[0,+\infty)$,
where
$$
\nu_{2,\epsilon}^\pm=\frac{-c-\sqrt{c^2+4d({\kappa_2}\mp{C_1^\pm}\epsilon)}}{2d},~~
0>\nu_1^+>\max\{\nu_1,\nu_{2,\epsilon}^+\},~~
\nu_1^-<\nu_1,
$$
and
$${C_1^+}=\mathop{\max}\limits_{[0,T]}{a_2}(t)p(t), ~~{C_1^-}=\mathop{\max}\limits_{[0,T]}{b_2}(t)q(t).
$$
\end{lemma}

\begin{proof}
Note that
$\frac{U(t,z)}{\phi_1(t)}$ and
$\frac{V(t,z)}{\phi_2(t)}$ are $T-$periodic in $t$
for any $z\in \mathbb R$, define
$$
\hat{u}(z)=\int_0^T{\frac{U(t,z)}{\phi_1(t)}dt},
\quad\hat{v}(z)=\int_0^T{\frac{V(t,z)}{\phi_2(t)}}dt
\text{ for any } z\in\mathbb{R}.
$$
Then direct calculations yield that
\begin{equation}\label{uv-hat}
\hspace{-0.7em}
\begin{cases}
\hat{u}_{zz}+c\hat{u}_z-\kappa_1\hat{u}+\int_0^T{\frac{b_1(t)q(t)V(t,z)}
{\phi_1(t)}dt}+\int_0^T{\frac{a_1(t)p(t){U^2}(t,z)-b_1(t)q(t)U(t,z)V(t,z)}{\phi_1(t)}dt}=0,\\
d\hat{v}_{zz}+c\hat{v}_z-\kappa_2\hat{v}+\int_0^T{\frac{a_2(t)p(t)U(t,z)V(t,z)
-b_2(t)q(t){V^2}(t,z)}{\phi_2(t)}dt}=0.
\end{cases}
\end{equation}
Since $\mathop{\lim}\limits_{z\to+\infty}(U(t,z),V(t,z))=(0,0)$ uniformly in $t\in\mathbb{R},$
for any $0<\epsilon<\min\left\{1,\frac{\kappa_2}{C_1^+}\right\},$
we can choose $M_\epsilon\gg1$ such that
$$(0,0)<(U(t,z),V(t,z))\leq(\epsilon,\epsilon)\text{ for any } (t,z)\in [0,T]\times [M_\epsilon,+\infty).
$$

Let ${V^+}(z)=\rho{e^{\nu_{2,\epsilon}^+z}}$
with $\nu_{2,\epsilon}^+=\frac{-c-\sqrt{{c^2}+4d({\kappa_2}-{C_1^+}\epsilon)}}{2d}<0$
and $\rho$ some positive constant,
then ${V^+}(z)$ is a solution of the linear equation
\begin{equation}\label{upper-eq}
d{v_{zz}}+c{v_z}-{\kappa_2}v +{C_1^+}\epsilon v=0.
\end{equation}
As $\hat{v}$ is bounded, we can choose $\rho>0$ large enough such that
$\hat{v}(M_\epsilon)\leq\rho{e^{\nu_{2,\epsilon}^+M_\epsilon}}$. In addition,
it follows from the second equation of \eqref{uv-hat} that
\begin{equation*}
\begin{aligned}
0&=d\hat{v}_{zz}+c\hat{v}_z-{\kappa_2}\hat{v}
+\int_0^T{\frac{a_2(t)p(t)U(t,z)V(t,z)-b_2(t)q(t){V^2}(t,z)}{\phi_2(t)}dt} \\
& \leq d\hat{v}_{zz}+ c\hat{v}_z-{\kappa_2}\hat{v}+\int_0^T{\frac{a_2(t)p(t)U(t,z)V(t,z)}{\phi_2(t)}dt}\\
&\leq d\hat{v}_{zz}+ c\hat{v}_z-{\kappa_2}\hat{v}+{C_1^+}\epsilon\hat{v}
\end{aligned}
\end{equation*}
for any $z\in [M_\epsilon,+\infty)$, and hence $\hat{v}(z)$ is a subsolution of \eqref{upper-eq} in $[M_\epsilon,+\infty).$
The maximum principle further yields that
$\hat{v}(z)\leq\rho{e^{\nu_{2,\epsilon}^+z}}$ for any $z\in [M_\epsilon,+\infty )$
by virtue of
$\mathop{\lim}\limits_{z\to+\infty}\hat{v}(z)=\mathop{\lim}\limits_{z\to+\infty}{V^+}(z)=0,$
which together with the Harnack inequality \eqref{har-estimate} and the definition of
$\hat{v}(z)$ implies that
there exists some $K_2^\prime>0$ such that
for any $(t,z)\in \mathbb{R}\times [0,+\infty),$
there hold $V(t,z)\leq{K_2^\prime}{e^{\nu_{2,\epsilon}^+z}}.$
Similarly, in view of
$$\int_0^T{\frac{{b_2(t)q(t){V^2}(t,z)}}{\phi_2(t)}dt}\leq{C_1^-}\epsilon\hat{v}
\text{ for any } z\in[M_\epsilon,+\infty),
$$
we can prove that
$V(t,z)\geq{K_2}{e^{\nu_{2,\epsilon}^-z}}$ for any
$(t,z)\in\mathbb{R}\times[0,+\infty)$
and some $K_2>0.$

We now consider $U(t,z).$
Note that $V(t,z)\leq{K_2^\prime}{e^{\nu_{2,\epsilon}^+z}}$
for any $(t,z)\in \mathbb{R}\times [0,+\infty),$
the Harnack inequality \eqref{har-estimate} implies that
there exist some $M_1,~M_2>0$ such that
\begin{equation*}
\int_0^T{\frac{b_1(t)q(t)V(t,z)}{\phi_1(t)}dt}\leq {M_1}{e^{\nu_{2,\epsilon}^+z}},
~~\int_0^T{\frac{a_1(t)p(t)U^2(t,z)}{\phi_1(t)}dt}\leq {M_2}(\hat{u}(z))^2
\end{equation*}
for any $z\in [0,+\infty)$. For any $0>\nu_1^+>\max\{\nu_1,\nu_{2,\epsilon}^+\},$ let
$L:=-{(\nu_1^+)^2}-c\nu_1^++\kappa_1>0.$
Then there exists some $M^\prime>0$ such that
${M_2}\hat{u}(z)\leq\frac{1}{2}\min\{L,\kappa_1\}$
for any $z\in[M^\prime,+\infty)$
since $\mathop{\lim}\limits_{z\to+\infty}\hat{u}(z)=0.$
It follows from the first equation of \eqref{uv-hat} that
\begin{equation*}
\begin{aligned}
0&=\hat{u}_{zz}+c\hat{u}_z-\kappa_1\hat{u}+\int_0^T{\frac{b_1(t)q(t)V(t,z)}{\phi_1(t)}dt}\\
&\quad+\int_0^T{\frac{a_1(t)p(t){U^2}(t,z)-b_1(t)q(t)U(t,z)V(t,z)}{\phi_1(t)}dt}\\
&\leq\hat{u}_{zz}+c\hat{u}_z-\kappa_1\hat{u}
+{M_1}{e^{\nu_{2,\epsilon}^+z}}+\frac{L}{2}\hat{u}
\end{aligned}
\end{equation*}
for any $z\in [M^\prime,+\infty),$
that is, $\hat{u}$ is a subsolution of equation
\begin{equation}\label{sup-eq}
-{u_{zz}}-c{u_z}+\kappa_1 u-\frac{L}{2}u-{M_1}{e^{\nu_{2,\epsilon}^+z}}=0
\text{ for any }z\in [M^\prime,+\infty).
\end{equation}
Let ${U^+}(z)=\delta{e^{\nu_1^+z}}$ with
$\delta\geq\frac{{2{M_1}}}{L}$ large enough such that
$\hat{u}(M^\prime)\leq\delta{e^{\nu_1^+M^\prime}},$ then
\begin{equation*}
\begin{aligned}
-&{U^+_{zz}}-c{U^+_z}+\kappa_1 U^+-\frac{L}{2}U^+-{M_1}{e^{\nu_{2,\epsilon}^+z}}\\
=&\left[-{(\nu_1^+)^2}-c\nu_1^++\kappa_1-\frac{L}{2}\right]\delta{e^{\nu_1^+z}}
-M_1e^{\nu_{2,\epsilon}^+z}\\
=&\frac{L}{2}\delta{e^{\nu_1^+z}}-M_1e^{\nu_{2,\epsilon}^+z}\geq 0
\end{aligned}
\end{equation*}
for any $z\in[M^\prime,+\infty),$
which implies that ${U^+}(z)$ is a supersolution of \eqref{sup-eq}.
The maximum principle then shows that
$\hat{u}(z)\leq \delta{e^{\nu_1^+z}}$ for any $z\in [M^\prime,+\infty).$
On the other hand, there exists ${M_3}>0$ such that
$$
\int_0^T {\frac{b_1(t)q(t)U(t,z)V(t,z)}{\phi_1(t)}dt}\leq {M_3}{e^{\nu_{2,\epsilon}^+z}}\hat{u} \text{ for any }  z\in [0,+\infty),
$$
we then see
\begin{equation*}
\begin{aligned}
0&=\hat{u}_{zz}+ c\hat{u}_z-\kappa_1\hat{u}+\int_0^T {\frac{b_1(t)q(t)V(t,z)}{\phi_1(t)}dt}\\
&\quad+\int_0^T{\frac{a_1(t)p(t){U^2}(t,z)-b_1(t)q(t)U(t,z)V(t,z)}{\phi_1(t)}dt}\\
&\geq\hat{u}_{zz}+c\hat{u}_z-\kappa_1\hat{u}-{M_3}{e^{\nu_{2,\epsilon}^+z}}\hat{u}
\end{aligned}
\end{equation*}
for any $z\in[0,+\infty),$
that is, $\hat{u}(z)$ is a supersolution of the following equation
\begin{equation}\label{super-eq}
{U_{zz}}+cU_z-\kappa_1U-{M_3}{e^{\nu_{2,\epsilon}^+z}}U=0 \text{ for any }  z\in [0,+\infty).
\end{equation}
Since ${(\nu_1^-)^2}+c\nu_1^--\kappa_1>0$ for any $\nu_1^-<\nu_1<0,$
we can choose $M^{\prime\prime}\geq M^\prime$ such that
$$
{e^{\nu_{2,\epsilon}^+M^{\prime\prime}}}\leq\frac{{(\nu_1^-)^2}+c\nu_1^--\kappa_1}{M_3}.
$$
Let ${U^-}(z)=\beta{e^{\nu_1^-z}}$ with $\beta$ satisfying $\beta{e^{\nu_1^-M^{\prime\prime}}}\leq\hat{u}(M^{\prime\prime}),$
then
\begin{equation*}
U^-_{zz}+cU^-_z-\kappa_1U^--{M_3}{e^{\nu_{2,\epsilon}^+z}}U^-
\geq\left[{(\nu_1^-)^2}+c\nu_1^--\kappa_1-{M_3}{e^{\nu_{2,\epsilon}^+M^{\prime\prime}}}\right]
\beta{e^{\nu_1^-z}}\geq 0
\end{equation*}
for any $z\in [M^{\prime\prime},+\infty),$
that is, ${U^-}(z)$ is a subsolution of \eqref{super-eq}.
Again using the maximum principle, we have
$\hat{u}(z)\geq \beta{e^{\nu_1^-z}}$
for any $z\in [M^{\prime\prime},+\infty).$
The same argument as above shows that there exist some $K_1,~K_1^\prime>0$ such that
${K_1}{e^{\nu_1^-z}}\leq U(t,z)\leq K_1^\prime{e^{\nu_1^+z}}$ for any
$(t,z)\in\mathbb R\times[0,+\infty).$ This ends all the proof.
\end{proof}

Denote
$$\kappa_3=\overline{b_1q-a_1p},
~\nu_3=\frac{-c+\sqrt{c^2+4\kappa_3}}{2}>0
\text{ and }
\kappa_4=\overline{b_2q},
~\nu_4=\frac{-c+\sqrt{c^2+4d\kappa_4}}{2d}>0.$$
Similarly, we can prove the a priori exponential estimates of
periodic traveling wave fronts of \eqref{uv-1-0} as $z\to-\infty$
as follows.

\begin{lemma}\label{general-estimate--}
Assume (A1)-(A3). Let $(U(t,z),V(t,z))$ be a solution of \eqref{UV}.
Then for any $\epsilon\in\left(0,\min\left\{1,\frac{\kappa_3}{C_2^+}\right\}\right),$
there exist positive constants $K_i,~K_i^\prime~(i=3,4)$ such that
\begin{equation}\label{general-estimate-UV-}
{K_3}{e^{\nu_{3,\epsilon}^-z}}\leq 1-U(t,z)\leq K_3^\prime{e^{\nu_{3,\epsilon}^+z}},
\quad {K_4}{e^{\nu_4^+z}}\leq 1-V(t,z)\leq K_4^\prime{e^{\nu_4^-z}},
\end{equation}
for any $(t,z)\in\mathbb{R}\times(-\infty,0],$ where
$$\nu_{3,\epsilon}^\pm=\frac{-c+\sqrt{c^2+4({\kappa_3}\mp{C_2^\pm}\epsilon)}}{2},
~ 0<\nu_4^-<\min\{\nu_4,\nu_{3,\epsilon}^+\},
~ \nu_4^+>\nu_4,
$$
and
$$
{C_2^+}=\mathop{\max}\limits_{[0,T]}{b_1}(t)q(t), ~
{C_2^-}=\mathop{\max}\limits_{[0,T]}{a_1}(t)p(t).
$$
\end{lemma}

\begin{remark}\label{approximation} \rm
The definitions of ${\nu_{i,\epsilon}^\pm}~(i=2,3)$ in Lemmas
\ref{general-estimate-+} and \ref{general-estimate--}
indicate that there exist
${\varepsilon_i^\pm}={\varepsilon_i^\pm}(\epsilon)$ such that for any
$\epsilon>0$ small enough,
there hold $\nu_{i,\epsilon}^\pm=\nu_i\pm{\varepsilon_i^\pm}$ with
${\varepsilon_i^\pm}={\varepsilon_i^\pm}(\epsilon)\to0$ as $\epsilon\to{0^+}.$
Actually, we know from the proof of Lemmas \ref{general-estimate-+} and \ref{general-estimate--}
that the a priori exponential estimates of the periodic traveling wave front
as it approaches the \textbf{stable} limiting state can only be characterized by the perturbations ${\nu_{i,\epsilon}^\pm}~(i=2,3)$ and $\nu_i^\pm~(i=1,4)$ rather than $\nu_i~(i=1,2,3,4),$ which is essentially different from the case that it approaches its \textbf{unstable} limiting states,
since there is no such exponential type sub-super solutions as
in \cite[Lemma 3.3]{zhao2011} that
equipped with $\nu_i~(i=1,2,3,4)$ as the decaying exponent for the a priori exponential estimates.
\end{remark}

\begin{lemma}\label{general-estimate-derives}
Assume (A1)-(A3). Let $(U(t,z),V(t,z))$ be a solution of \eqref{UV}.
Then there exist $C_1,~C_2>0$ such that
\begin{align}
\left|{W(t,z)}\right|+\left|{{W_z}(t,z)}\right|+\left|{{W_{zz}}(t,z)}\right|
\leq C_1{e^{\nu_4^-z}} \text{ for any } (t,z)\in R\times (-\infty,0],
\label{general-estimate-derives--}
\\
\left|{W(t,z)-1}\right|+\left|{{W_z}(t,z)}\right|+\left|{{W_{zz}}(t,z)}\right|
\leq{C_2}{e^{\nu_1^+z}} \text{ for any } (t,z)\in R\times [0,+\infty),
\label{general-estimate-derives-+}
\end{align}
where $W(t,z)=U(t,z)$ or $V(t,z),$
$\nu_1^+$ and $\nu_4^-$ are defined in Lemmas \ref{general-estimate-+} and \ref{general-estimate--}.
\end{lemma}

\begin{proof}
The proof is similar to that of \cite[Proposition $3.4$]{zhao2011},
using the interior parabolic estimates and the Harnack inequality \eqref{har-estimate},
so we omit it.
\end{proof}

Next we establish the exact exponential decay rate of
periodic traveling wave fronts of \eqref{uv-1-0} near
the limiting state $(0,0).$
Denote $(U(t,z),V(t,z))$ by $(u(t,z),v(t,z))$
for the convenience of writing in the current section.

\begin{proposition}\label{1<2}
Let $(u(t,z),v(t,z))$ be a periodic traveling wave front of \eqref{uv-1-0}.
Then
\begin{equation*}
\mathop{\lim}\limits_{z\to+\infty}\frac{v(t,z)}{{k_1}{e^{\nu_2z}}\phi_2(t)}=1 \text{ and }
~\mathop{\lim}\limits_{z\to+\infty}\frac{v_z(t,z)}{{k_1}{e^{\nu_2z}}\phi_2(t)}=\nu_2
\text{ uniformly in } t\in \mathbb R.
\end{equation*}
If in addition $\nu_1<\nu_2,$ then
\begin{equation*}
\mathop{\lim}\limits_{z\to+\infty}\frac{u(t,z)}{{k_1}{e^{\nu_2z}}\tilde{\phi}_1(t)}=1
\text{ and }
~\mathop{\lim}\limits_{z\to+\infty}\frac{{u_z}{(t,z)}}{{k_1}{e^{\nu_2z}}\tilde{\phi}_1(t)}=\nu_2
\text{ uniformly in } t\in \mathbb R,
\end{equation*}
where $k_1>0$ is some constant and
\begin{equation}\label{phi-1-t}
\begin{cases}
\tilde{\phi}_1(t)=\tilde{\phi}_1(0){e^{\int_0^t{(\upsilon_1-a_1(s)p(s))ds}}}
+\int_0^t{{e^{\int_s^t{(\upsilon_1-a_1(\tau)p(\tau))d\tau}}}b_1(s)q(s)\phi_2(s)ds},\\
\tilde{\phi}_1(0)={\left({1-{e^{\int_0^T{(\upsilon_1-a_1(s)p(s))ds}}}}\right)^{-1}}
\int_0^T{{e^{\int_s^T{(\upsilon_1-a_1(\tau)p(\tau))d\tau}}}b_1(s)q(s)\phi_2(s)ds}
\end{cases}
\end{equation}
with $\upsilon_1=\nu_2^2+c\nu_2.$
\end{proposition}

\begin{proof}
Consider the linearized system of \eqref{uv-1-0} at $(0,0)$
\begin{equation*}
\begin{cases}
u_t=u_{zz}+cu_z-a_1pu+b_1qv,\\
v_t=dv_{zz}+cv_z+(b_2q-a_2p)v.
\end{cases}
\end{equation*}
Since the $v-$equation is not coupled with $u,$
the exponential behavior of $v$ can be obtained by
employing the two-sided Laplace transform method.
Indeed, according to the proof of \cite[Theorem 2.7]{du2016}, there exist constants
$K^\prime,~K^{\prime\prime}>0$
and $\varepsilon_0>0$ such that for any $0<\varepsilon<\varepsilon_0,$
$$
\eta(t,z):=v(t,z)-{k_1}{e^{\nu_2z}}\phi_2(t)\text{ and } \tilde{\eta}(t,z):=v_z(t,z)-k_1\nu_2{e^{\nu_2z}}\phi_2(t)
$$
satisfy
\begin{equation}\label{v-higer}
\mathop{\sup}\limits_{t\in[0,T]}\left|{\eta (t,z)}\right|\leq{K^\prime}{e^{(\nu_2-{\varepsilon})z}}\text{ and } \mathop{\sup}\limits_{t\in[0,T]}\left|{\tilde{\eta} (t,z)}\right|\leq{K^{\prime\prime}}{e^{(\nu_2-{\varepsilon})z}}
\text{ for any } z\geq 0,
\end{equation}
respectively. Therefore, we have
\begin{equation*}
\mathop{\lim}\limits_{z\to+\infty}\frac{v(t,z)}{{k_1}{e^{\nu_2z}}\phi_2(t)}=1 \text{ and } ~\mathop{\lim}\limits_{z\to+\infty}\frac{v_z(t,z)}{{k_1}{e^{\nu_2z}}\phi_2(t)}=\nu_2
\text{ uniformly in } t\in \mathbb R.
\end{equation*}

Now we study the exponential behavior of $u$ as $z\to+\infty.$
It follows from $\nu_1<\nu_2<0$ that $\nu_2^2+c\nu_2-\kappa_1<0,$
then the equation
$$b_1(t)q(t)\phi_2(t)+[{\nu_2^2}+c\nu_2-a_1(t)p(t)]w-{w_t}=0$$
admits a unique positive periodic solution $\tilde{\phi}_1(t)$ given by \eqref{phi-1-t}.
A direct calculation shows that
$\omega(t,z):={k_1}{e^{\nu_2z}}\tilde{\phi}_1(t)$ satisfies
$$b_1(t)q(t){k_1}{e^{\nu_2z}}\phi_2(t)
-a_1(t)p(t)\omega+{\omega_{zz}}+c{\omega_z}-{\omega_t}=0.$$
Let
$$
\xi(t,z)=\frac{u(t,z)-{k_1}{e^{\nu_2z}}\tilde{\phi}_1(t)}{\phi_1(t)} \text{ and }
\zeta(t,z)=\frac{v(t,z)-{k_1}{e^{\nu_2z}}\phi_2(t)}{\phi_1(t)}
$$
for any $(t,z)\in\mathbb R\times[0,+\infty),$ then
$$
R(t,z)-\kappa_1 \xi+{\xi_{zz}}+c{\xi_z}-{\xi_t}=0 \text{ for any }
(t,z)\in\mathbb R\times[0,+\infty),
$$
where $$R(t,z)=\left[a_1(t)p(t)u^2(t,z)-b_1(t)q(t)u(t,z)v(t,z)\right]
{\phi_1^{-1}(t)}+b_1(t)q(t)\zeta(t,z).$$
Note that
$\mathop{\sup}\limits_{t\in\mathbb R}\left|{\zeta(t,z)}\right|
=O({e^{(\nu_2-\varepsilon)z}})$ as $z\to+\infty.$
In addition, it follows from \eqref{general-estimate-derives-+} that
$$
\mathop{\sup}\limits_{t\in R}\left[a_1(t)p(t)u^2(t,z)-b_1(t)q(t)u(t,z)v(t,z)\right]
=O\left({e^{2\nu_1^+z}}\right)\text { as } z\to+\infty.
$$
Since we can take $0>\nu_1^+>\max\{\nu_1,\nu_{2,\epsilon}^+\}$ small enough such that
$2\nu_1^+<\nu_2-\varepsilon,$
there exist positive constants $M$ and $K_M$ such that
$$
\left|\left[a_1(t)p(t)u^2(t,z)-b_1(t)q(t)u(t,z)v(t,z)\right]{\phi_1^{-1}(t)}\right|+
\left|{b_1(t)q(t)\zeta(t,z)}\right|\leq{K_M}{e^{(\nu_2-{\varepsilon})z}}
$$
for any $(t,z)\in \mathbb R\times[M,+\infty).$

Next we show that
$\mathop{\sup}\limits_{t\in\mathbb R}\left|{\xi(t,z)}\right|=o({e^{{\nu_2}z}})$
as $z\to+\infty.$
In view of \eqref{general-estimate-derives-+}, there holds
$\mathop{\sup}\limits_{t\in\mathbb R}\left|{\xi(t,z)}\right|=O\left({e^{\nu_1^+z}}\right)
\text{ as } z\to+\infty.$
Noting that $\nu_1<\nu_2,$
let $0<\varepsilon<\varepsilon_0$ be small enough such that
$\nu_2-{\varepsilon}>\nu_1$ and hence
$Q:={(\nu_2-{\varepsilon})^2}+c(\nu_2-{\varepsilon})-\kappa_1<0.$
It is easy to verify that
$\pm K{e^{(\nu_2-{\varepsilon})z}}$
satisfy respectively
$$
R(t,z)-\kappa_1\omega+{\omega_{zz}}+c{\omega_z}-{\omega_t}\leq(\geq)~0
\text{ for any } (t,z)\in\mathbb R\times[M,+\infty)
$$
whenever $K\geq\frac{K_M}{\left|Q\right|}.$
Since $\left|{\xi(t,z)}\right|$ is bounded in $(t,z)\in\mathbb R\times{\mathbb R^+},$
there exists
${K_Q}\geq\frac{K_M}{\left|Q\right|}$ such that
$\left|{\xi(t,M)}\right|\leq{K_Q}{e^{(\nu_2-{\varepsilon})M}}$
for any $t\in \mathbb R,$
then
\begin{equation}\label{upper-sub}
-{K_Q}{e^{(\nu_2-{\varepsilon})z}}\leq\xi(t,z)
\leq{K_Q}{e^{(\nu_2-{\varepsilon})z}}
\text{ for any } (t,z)\in \mathbb R\times[M,+\infty).
\end{equation}
Indeed, set
$$
{\omega^\pm}(t,z)=\pm{K_Q}{e^{(\nu_2-{\varepsilon})z}}-\xi(t,z)
\text{ for any } (t,z)\in \mathbb R\times[M,+\infty),
$$
then we obtain
\begin{equation}\label{upper-sub-ineq}
\omega_{zz}^++c\omega_z^+-\omega_t^+-\kappa_1{\omega^+}\leq 0,\\
\quad\omega_{zz}^-+c\omega_z^--\omega_t^--\kappa_1{\omega^-}\geq 0.
\end{equation}
Notice that ${\omega^\pm}(t,z)$ is $T-$ periodic in $t,$
it is sufficient to show that
${\omega^+}(t,z)\geq0$ for any $(t,z)\in(0,2T)\times[M,+\infty),$
while a similar argument holds for ${\omega^-}(t,z)\leq 0.$
Assume on the contrary that
$$
\mathop{\inf}\limits_{(t,z)\in(0,2T)\times [M,+\infty)}{\omega^+}(t,z)<0,
$$
it follows from
$\mathop{\lim}\limits_{z\to+\infty}\mathop{\sup}\limits_{t\in[0,2T]}{\omega^+}(t,z)=0$
that there exists $({t^*},{z^*})\in(0,2T)\times(M,+\infty )$ such that
$${\omega^+}({t^*},{z^*})=\mathop {\inf}\limits_{(t,z)\in(0,2T)\times [M,+\infty)}{\omega^+}(t,z)<0,
$$
and then
$\left[{\omega_{zz}^++c\omega_z^+-\omega_t^+-\kappa_1{\omega^+}}\right]\big|_{({t^*},{z^*})}>0,$
which contradicts to \eqref{upper-sub-ineq}. Hence \eqref{upper-sub} implies that
$$\mathop{\sup}\limits_{t\in\mathbb R}\left|{\xi(t,z)}\right|=o({e^{{\nu_2}z}})
\text{ as } z\to+\infty.
$$
Therefore, we see from the definition of $\xi(t,z)$ that
$$\mathop{\lim}\limits_{z\to +\infty}\frac{u(t,z)}{{k_1}{e^{\nu_2z}}\tilde{\phi}_1(t)}=1
\text{ uniformly in } t\in \mathbb R.$$

The argument for $u_z$ is similar and we only give a brief sketch here. Let
\begin{equation*}
\tilde\xi(t,z)=\frac{u_z(t,z)-{k_1}{\nu_2}{e^{\nu_2z}}\tilde{\phi}_1(t)}{\phi_1(t)} \text{ and }
~\tilde\zeta(t,z)=\frac{v_z(t,z)-{k_1}{\nu_2}{e^{\nu_2z}}\phi_2(t)}{\phi_1(t)}
\end{equation*}
for any $(t,z)\in\mathbb R\times[0,+\infty)$, then
$$\tilde R(t,z)-\kappa_1\tilde\xi+{\tilde\xi_{zz}}+c{\tilde\xi_z}-{\tilde\xi_t}=0
\text{ for any } (t,z)\in\mathbb R\times[0,+\infty),
$$
where
$\tilde R(t,z)=[(2a_1pu-b_1qv)u_z-(b_1qu)v_z]{\phi_1^{-1}}+b_1q\tilde\zeta.$
The same argument as above shows that
$$
\mathop{\sup}\limits_{t\in\mathbb R}\left|{\tilde\xi(t,z)}\right|=o({e^{\nu_2z}})
\text{ as } z\to+\infty,
$$
and hence
$$
\mathop{\lim}\limits_{z\to+\infty}\frac{{u_z}{(t,z)}}{{k_1}{e^{\nu_2z}}\tilde\phi_1(t)}=\nu_2
\text{ uniformly in } t\in\mathbb R.
$$
We now complete all the proof.
\end{proof}

\begin{remark} \rm
Note that $\nu_1<\nu_2$ has nothing to do with the exponential behavior of $v$ as $z\to+\infty$, but it is very important for that of $u$. In fact, in our recent paper \cite{du2016} concerning system \eqref{uv-eq} with \textbf{monostable} structure, we have proposed a sufficient condition:
\begin{description}
\item[(C3)] $b_1(t)q(t)\leq a_1(t)p(t)<5~b_1(t)q(t)$ for any $t\in\mathbb R,$
\end{description}
to ensure that $\nu_1<\nu_2$, see \cite[Theorem 2.7]{du2016}.
\end{remark}

In the following, we discuss the other two cases, that is,
the exponential behaviors of $u$ as $z\to+\infty$ for
$\nu_1>\nu_2$ and for $\nu_1=\nu_2.$
When $\nu_1>\nu_2$, we regard $z$ as the evolution variable and then
employ the Laplace transform method and spectrum analysis to address this issue.

Let $Y=L_T^2\times L_T^2$ with
$$
L_T^2:=\left\{\int_0^T{{{\left|{h(t+s)}\right|}^2}ds<\infty},~h(t+T)=h(t)\right\}
\text{ and }
{\left\|h\right\|_{L_T^2}}={\left(\int_0^T{{{\left|{h(s)}\right|}^2}ds}\right)^{\frac{1}{2}}},
$$
and
$$
H_T^1=\left\{h\in L_T^2,~\mathop{\sup}\limits_{t\in\mathbb R}
\int_0^T|h^\prime(t+s)|^2ds<\infty\right\}.
$$
Define an operator $\mathcal A:D(\mathcal A)\subset Y\rightarrow Y$ as
\begin{equation}\label{operatorA}
\mathcal A =\left({\begin{array}{*{20}{c}}0&I \\{{\partial_t}-{g_u}(t,0,0)}&{-c}\end{array}}\right).
\end{equation}
It is easy to verify that $\mathcal A$ is closed and densely defined
in $D(\mathcal A)=H_T^1\times L_T^2$, $\nu_1\in\sigma(\mathcal A)=\sigma_p(\mathcal A)$ and
\begin{equation*}
{\rm ker}(\nu_1I-\mathcal A)^n={\rm ker}(\nu_1I-\mathcal A)
={\rm span}\left\{{\left({\begin{array}{*{20}{c}}{\phi_1}
\\{\nu_1\phi_1}\end{array}}\right)}\right\}
\text{ for } n=2,3,\cdots,
\end{equation*}
which implies that $\nu_1$ is a simple pole of
${(\lambda I-\mathcal A)^{-1}}$~(see \cite[Remark A.2.4]{lunardi1995}).
Moreover, a similar argument to \cite[Proposition 3.6]{zhao2011} shows that
there exists some $\varepsilon^\prime>0$ such that
${\Theta_{{\varepsilon^\prime}}}\cap\sigma(\mathcal A)=\{\nu_1\},$
where
${\Theta_{{\varepsilon^\prime}}}=\{\lambda\in\mathbb C|\nu_1-{\varepsilon^\prime}
\leq\operatorname{Re}\lambda\leq\nu_1+{\varepsilon^\prime}\}$
is the vertical strip containing the vertical line $Re\lambda=\nu_1.$
Thus,
$\nu_1$ is the only singular point of $\lambda I-\mathcal A$
in $\Theta_{{\varepsilon^\prime}}.$
Then by \cite{lunardi1995}, the Laurent series of $(\lambda I-\mathcal A)^{-1}$ near $\lambda=\nu_1$ is given as
\begin{equation}\label{laurent}
(\lambda I-\mathcal A)^{-1}
=\sum\limits_{n=0}^\infty{{(-1)^n}{{(\lambda-\nu_1)}^n}{S^{n+1}}}
+\frac{P}{\lambda-\nu_1}
+\sum\limits_{n=1}^\infty{{{(\lambda-\nu_1)}^{-n-1}}{D^n}},
\end{equation}
where
$$
P=\frac{1}{2\pi i}\int_\Gamma{{(\lambda I-\mathcal A)}^{-1}}d\lambda
$$
is the spectral projection with
$\Gamma:\left|{\lambda-\nu_1}\right|<\varepsilon^{\prime\prime}$ for some
$\varepsilon^{\prime\prime}>0$ small enough, and
$$
S=\frac{1}{2\pi i}\int_\Gamma\frac{{(\lambda I-\mathcal A)}^{-1}}{\lambda-\nu_1}d\lambda
=\mathop{\lim}\limits_{\lambda\to\nu_1}(I-P){(\lambda I-\mathcal A)}^{-1},
~~D=(\mathcal A-\nu_1I)P.
$$
Since $\nu_1$ is a simple pole of ${(\lambda I-\mathcal A)^{-1}},$
it follows from \cite[Proposition $A.2.2$]{lunardi1995} that
${R(P)}={\rm ker}(\nu_1I-\mathcal A)$ and
hence ${{D^n}=0}$ for any $n\in N^+.$
Then \eqref{laurent} becomes
\begin{equation*}
(\lambda I-\mathcal A)^{-1}
=\sum\limits_{n=0}^\infty{{(-1)^n}{{(\lambda-\nu_1)}^n}{S^{n+1}}}
+\frac{P}{\lambda-\nu_1},
\end{equation*}
with projection $P$ the residue of
$(\lambda I-\mathcal A)^{-1}$ at $\lambda=\nu_1.$

Let
$\lambda=\mu+i\eta\in\rho(\mathcal A)$ with $\mu,~\eta\in\mathbb R,$
denote
$$
S=\left\{{\left(\left.{\begin{array}{*{20}{c}}0\\j \end{array}}\right)\right|j\in L_T^2} \right\}\subset Y
$$
and $(\lambda I-\mathcal A)_S^{-1}$ the
restriction of $(\lambda I-\mathcal A)^{-1}$ to $S,$
then it is not difficult to estimate that
there exist positive constants $C$ and $\varrho$ such that
\begin{equation}\label{resolvent-es}
\left\|{(\lambda I-\mathcal A)_S^{-1}}\right\|\leq\frac{C}{{\left|\eta\right|}}
\text{ for any }
\mu\in[\nu_1-{\varepsilon^\prime},\nu_1+{\varepsilon^\prime}],
~\left|\eta\right|\geq\varrho.
\end{equation}

We now state the exponential behavior of $u$ as $z\to+\infty$
when $\nu_1>\nu_2.$
\begin{theorem}\label{1>2}
Assume (A1)-(A3). Let $(u(t,z),v(t,z))$ be a solution of
\eqref{UV}. If $\nu_1>\nu_2,$ then
\begin{equation*}
\mathop{\lim}\limits_{z\to+\infty}\frac{u(t,z)}{{k_2}{e^{\nu_1z}}\phi_1(t)}=1,
~\mathop{\lim}\limits_{z\to+\infty}\frac{u_z(t,z)}{{k_2}{e^{\nu_1z}}\phi_1(t)}=\nu_1
 \text{ uniformly in } t\in \mathbb R,
\end{equation*}
where $k_2>0$ is some constant.
\end{theorem}

\begin{proof}
Introduce an auxiliary function
$$
\left\{
{\chi\in C_b^3(R,R)\text{ with~~}
{\begin{gathered}
{ \chi(z)\equiv 1,~z \geq 0;}\hfill\\
{ \chi(z)\equiv 0,~z <-1;} \hfill\\
{ \left|{{\chi^\prime}}\right|+\left|{{\chi^{\prime\prime}}}\right|
+\left|{{\chi^{\prime\prime\prime}}}\right|<\infty
~\text{for any } z\in\mathbb R}
\end{gathered}}}
\right\}.
$$
Setting
$w={u_z},~\breve{u}=\chi u,~\breve{w}={(\chi u)_z},$
then a direct calculation shows that
\begin{equation}\label{wwu}
\breve{w}_z+c\breve{w}=\breve{u}_t-{g_u}(t,0,0)\breve{u}+\chi[{g_u}(t,0,0)u-g(t,u,v)]
 +{\chi^{\prime\prime}}u+2{\chi^\prime}{u_z}+c{\chi^\prime}u.
\end{equation}
Denote
$$
\tilde g(t,z)=\chi[{g_u}(t,0,0)u-g(t,u,v)]
 +{\chi^{\prime\prime}}u+2{\chi^\prime}{u_z}+c{\chi^\prime}u,
 $$
then \eqref{wwu} can be rewritten as the following first order system
\begin{equation}\label{ODEs}
\frac{d}{{dz}}\left({\begin{array}{*{20}{c}}
  {{\breve{u}}}\\{{\breve{w}}}\end{array}}\right)=\mathcal A \left({\begin{array}{*{20}{c}}
  {{\breve{u}}}\\{{\breve{w}}}\end{array}}\right)+\left({\begin{array}{*{20}{c}}0 \\
  {\tilde g(t,z)}\end{array}} \right),
\end{equation}
where $\mathcal A$ is defined in \eqref{operatorA}.

Let
$0<\varepsilon^\prime<\min\left\{-\frac{\nu_1}{3},\frac{3}{4}(\nu_1-\nu_2)\right\}$
be sufficiently small such that
$\Theta_{\varepsilon^\prime}\cap\sigma(\mathcal A)=\{\nu_1\}.$
For this fixed $\varepsilon^\prime>0,$
it follows from $\nu_1>\nu_2$ and Remark \ref{approximation} that
there exists a small enough
$0<\epsilon<\min\left\{1,\frac{\kappa_2}{C_1^+}\right\}$
such that
$\nu_{2,\epsilon}^+=\nu_2+\varepsilon_2^+(\epsilon)<\nu_1,$
where $\nu_{2,\epsilon}^+$ is defined in
Lemma \ref{general-estimate-+}.
Therefore, we can take $\nu_1^+$ with
$0>\nu_1^+>\max\{\nu_1,\nu_{2,\epsilon}^+\}=\nu_1$ such that $0<\varepsilon^+:=\nu_1^+-\nu_1<\frac{1}{2}\varepsilon^\prime.$

In view of \eqref{general-estimate-derives-+}, we see
$\mathop{\sup}_{t\in R}(\left|{{\breve{u}}}\right|
+\left|{{\breve{w}}}\right|+\left|{\breve{u}_z}\right|
+\left|{\breve{w}_z}\right|)=O({e^{\nu_1^+z}})$
as $z\to+\infty.$
Thus for any $Re\lambda\in(\nu_1^+,\nu_1+{\varepsilon^\prime}],$
there holds
$({e^{-\lambda z}}{\breve{u}},{e^{-\lambda z}}{\breve{w}})
\in {W^{1,1}}(\mathbb R,Y)\cap{W^{1,\infty}}(\mathbb R,Y).
$
Now taking the two-sided Laplace transform of \eqref{ODEs} with respect to $z,$
we obtain
\begin{equation}\label{laplace-eq}
\left({\begin{array}{*{20}{c}}
{\int_\mathbb R{{e^{-\lambda s}}{\breve{u}}(\cdot,s)ds}}
\\{\int_\mathbb R{{e^{-\lambda s}}{\breve{w}}(\cdot,s)ds}}
\end{array}}\right)=:\mathcal F(\lambda)={(\lambda I-\mathcal A)^{-1}}\left({\begin{array}{*{20}{c}} 0 \\
{\int_\mathbb R{{e^{-\lambda s}}\tilde g(\cdot,s)ds}}\end{array}}\right),
\end{equation}
where $\nu_1^+<Re\lambda\leq\nu_1+{\varepsilon^\prime}.$
It then follows from Lemma \ref{general-estimate-derives} and Proposition \ref{1<2} that
$$
\left|{{g_u}(t,0,0)u-g(t,u,v)}\right|=O(\left|{v}\right|
+\left|{u}\right|^2)=O(e^{\sigma z})
\text{ as } z\to+\infty,
$$
where $\sigma=\max\{\nu_2,2\nu_1^+\}<\nu_1.$
Recall the definition of $\varepsilon^\prime,$
there is $\nu_1-\frac{4}{3}\varepsilon^\prime\geq\sigma.$ Then
$$
\mathop{\sup}\limits_{t\in\mathbb R}\left({\left|{\tilde g}\right|+\left|{{\tilde g}_z}\right|}\right)=O\left({e^{(\nu_1-\frac{4}{3}\varepsilon^\prime)z}}\right)
\text{ as }z\to+\infty.$$
Therefore,
${\int_\mathbb R{{e^{-\lambda s}}\tilde g(\cdot,s)ds}}$ and
${\int_\mathbb R{{e^{-\lambda s}}\widetilde g_z(\cdot,s)ds}}$
are analytic for $\lambda$ with $Re\lambda\in(\nu_1-\frac{4}{3}\varepsilon^\prime,0).$
Let $\lambda=\mu+i\eta,$ then
$\int_\mathbb R{{e^{-\lambda s}}\tilde g(\cdot,s)ds}
=\int_\mathbb R{{e^{-i\eta s}}\cdot{e^{-\mu s}}\tilde g(\cdot,s)ds}=\hat{f_\mu}(\eta),$
where $\hat{f_\mu}$ is the Fourier transform of
${f_\mu}(s):={e^{-\mu s}}\tilde g(\cdot,s).$
It is easy to see that
$${f_\mu}(s)\in{W^{1,1}}(\mathbb R,L_T^2)\cap{W^{1,\infty}}(\mathbb R,L_T^2)
\text{ for any fixed } \mu\in\left[\nu_1-\varepsilon^\prime,-\frac{1}{2}{\varepsilon^\prime}\right].
$$
Particularly,
${\left\|{{e^{-\mu s}}\tilde g}\right\|_{{W^{1,1}}(\mathbb R,L_T^2)}}$
is uniformly bounded
in $\mu\in[\nu_1-\varepsilon^\prime,-\frac{1}{2}{\varepsilon^\prime}].$
Then there exist $C_1>0$ and $\varrho_1>0$ such that
for any $\left|\eta\right|\geq{\varrho_1},$
$$
{\left\|{\hat{f}(\eta)}\right\|_{L_T^2}}={\left\|{\int_\mathbb R{{e^{-\lambda s}}
\tilde g(\cdot,s)ds}}\right\|_{L_T^2}}
\leq\frac{C_1}{{\left|\eta\right|}}\text{ whenever } \mu\in\left[\nu_1-\varepsilon^\prime,-\frac{1}{2}{\varepsilon^\prime}\right].
$$
Moreover, by \eqref{resolvent-es}, there exist $C_2>0$ and $\varrho_2>0$ such that
\begin{equation}\label{resolvent-ineq}
\left\|{{(\lambda I-\mathcal A)^{-1}}G(\lambda)}\right\|_Y\leq\frac{C_2}{{\left|\eta\right|}^2}
\text{ for any } \left|\eta\right|\geq\varrho_2,
~\mu\in[\nu_1-\varepsilon^\prime,\nu_1+\frac{1}{2}{\varepsilon^\prime}]\setminus\{\nu_1\},
\end{equation}
where
$$
G(\lambda)=\left({\begin{array}{*{20}{c}}0\\{\int_\mathbb R{{e^{-\lambda s}}\tilde g(\cdot,s)ds}}\end{array}}\right).
$$
Thus
$
\mathcal F(\lambda)=\mathcal F(\mu+i\eta)\in{L^1}(\mathbb R,Y)\cap{L^\infty}(\mathbb R,Y)
$
for any $\mu\in\left[\nu_1-\varepsilon^\prime,\nu_1+\frac{1}{2}{\varepsilon^\prime}\right]
\setminus\{\nu_1\}.$

Let $\mu=\nu_1+\frac{1}{2}{\varepsilon^\prime},$
the inverse Laplace transform implies that
\begin{equation}\label{inverse-laplace}
\left({\begin{array}{*{20}{c}}{{u}(\cdot,z)}\\{{w}(\cdot,z)}
\end{array}}\right)=\left({\begin{array}{*{20}{c}}
{{\breve{u}}(\cdot,z)}\\{{\breve{w}}(\cdot,z)}
\end{array}}\right)
=\frac{1}{{2\pi i}}\int_{\mu-i\infty}^{\mu+i\infty}{{e^{\lambda z}}
{{(\lambda I-\mathcal A)}^{-1}}G(\lambda)}d\lambda
\text{ for any } z\geq 0.
\end{equation}
By virtue of \eqref{resolvent-ineq}, there holds that
$$
\mathop{\lim}\limits_{\left|\eta\right|\to\infty}
{\int_{\nu_1-{\varepsilon^\prime}}^\mu{\left\|{{e^{(\tau+i\eta)z}}{{((\tau+i\eta )I
-\mathcal A)}^{-1}}G(\tau+i\eta)}\right\|}_Y}d\tau=0
\text{ for any } z\geq 0.
$$
Therefore, the path of the integral in \eqref{inverse-laplace} can be shifted to
$Re\lambda=\nu_1-\varepsilon^\prime$ such that
\begin{equation}\label{uw-expression}
\begin{aligned}
\left({\begin{array}{*{20}{c}}{{u}(\cdot,z)}\\{{w}(\cdot,z)}\end{array}}\right)
=\frac{1}{{2\pi i}}\int_{\nu_1-\varepsilon^\prime-i\infty}
^{\nu_1-\varepsilon^\prime+i\infty}
{{e^{\lambda z}}{{(\lambda I-\mathcal A)}^{-1}}G(\lambda)}d\lambda
+Res(e^{\lambda z}\mathcal F(\lambda),\nu_1)
\end{aligned}
\end{equation}
for any $z\geq 0,$
where
$
Res(g,\lambda_0):=\frac{1}{{2\pi i}}\int_{\Gamma:\left|{\lambda-\lambda_0}\right|
<{\varepsilon^{\prime\prime}}}{g(\lambda)d\lambda}
$
denotes the residue of $g$ at $\lambda_0.$
Furthermore, with the aid of
$$
{(\lambda I-\mathcal A)^{-1}G(\lambda)}=
\sum\limits_{n=0}^\infty{{(-1)^n}{{(\lambda-\nu_1)}^n}{S^{n+1}}G(\lambda)}
+\frac{PG(\nu_1)}{\lambda-\nu_1}
-\frac{P[G(\nu_1)-G(\lambda)]}{\lambda-\nu_1}
$$
for $\left|{\lambda-\nu_1}\right|<\varepsilon^{\prime\prime}$
with some $\varepsilon^{\prime\prime}$ small enough,
$$
PG\subset{\rm ker}(\nu_1I-\mathcal A)
={\rm span}\left\{{\left({\begin{array}{*{20}{c}}{\phi_1}
\\{\nu_1\phi_1}\end{array}}\right)}\right\}
$$
and $G(\lambda)$ is analytic in
$Re\lambda\in\left(\nu_1-\frac{4}{3}\varepsilon^\prime,0\right),$
the residue theorem implies that
\begin{equation}\label{uw-exact-expression}
\begin{aligned}
\left({\begin{array}{*{20}{c}}{u(t,z)}\\{w(t,z)}\end{array}}\right)=
&\frac{e^{(\nu_1-{\varepsilon^\prime})z}}{2\pi}\int_{-\infty}^{+\infty}
{{e^{i\eta z}}{{((\nu_1-{\varepsilon^\prime}+i\eta)I
-\mathcal A)}^{-1}}G(\nu_1-{\varepsilon^\prime}+i\eta)}d\eta \\
&+{k_2}{e^{\nu_1z}}\left({\begin{array}{*{20}{c}}{\phi_1(t)}\\
{\nu_1\phi_1(t)}\end{array}}\right)
 \text{ for any } z\geq 0,
\end{aligned}
\end{equation}
where $k_2 \geq 0$ is some constant.

Let $\zeta(t,z)=u(t,z)-{k_2}{e^{\nu_1z}}\phi_1(t)$
for any $(t,z)\in\mathbb R\times{\mathbb R^+}.$
It follows from \eqref{uw-exact-expression} and \eqref{resolvent-ineq} that
there exists $C_3>0$ such that
$$
{\left({\int_{z-1}^{z+1}}{{\int_0^{2T}
{{{\left|{\zeta(\tau,s)}\right|}^2}d\tau ds}}}\right)^{\frac{1}{2}}}
\leq{C_3}{e^{(\nu_1-{\varepsilon^\prime})z}}
\text{ for any } z\geq 0.
$$
Note that $\zeta(t,z)$ satisfies
$$
[g(t,u,v)-{g_u}(t,0,0)u]+{g_u}(t,0,0)\zeta+{\zeta_{zz}}+c{\zeta_z}-{\zeta_t}=0
\text{ for any } (t,z)\in R\times{R^+},
$$
and
$$
|g(t,u,v)-{g_u}(t,0,0)u|=O({e^{\sigma z}}) \text{ as } z\to+\infty,
$$
the interior parabolic estimates then yield that
there exists $C_4>0$ independent of $z$ such that
$$
\left({\int_{z-\frac{1}{2}}^{z+\frac{1}{2}}{\int_T^{2T}
{\left({{{\left|{{\zeta_{zz}}(\tau,s)}\right|}^2}+{{\left|{{\zeta_z}(\tau,s)}\right|}^2}
+{{\left|{{\zeta_t}(\tau,s)}\right|}^2}}\right)d\tau ds}}}\right)^{\frac{1}{2}}
\leq{C_4}{e^{\gamma z}}
\text{ for any } z\geq 0,
$$
where
$
\gamma=\max\{\nu_1-\varepsilon^\prime,\sigma\}
=\max\{\nu_1-\varepsilon^\prime,\nu_2\}<\nu_1.
$
It follows from the Sobolev embedding theorem that there exists $C_5>0$ such that
$$
\mathop{\sup}\limits_{t\in[0,T]}\left|{\zeta(t,z)}\right|\leq{C_5}{e^{\gamma z}}
\text{ for any } z\geq 0.
$$
Since $\nu_1^-<\nu_1$ is arbitrary,
we can take some $\nu_1^-$ such that $\nu_1^-\geq\gamma.$ Using
Lemma \ref{general-estimate-+}, we then see $k_2>0.$
Consequently,
$$
\mathop{\lim}\limits_{z\to+\infty}\frac{u(t,z)}{k_2{e^{\nu_1z}}\phi_1(t)}=1
 \text{ uniformly in } t\in \mathbb R.
$$

Similarly, let
$$
\tilde\zeta(t,z)={u_z}(t,z)-{k_2}\nu_1{e^{\nu_1z}}\phi_1(t) \text{ for any } (t,z)\in \mathbb R\times\mathbb R^+.
$$
It then follows from \eqref{resolvent-ineq} that there exits $C_6>0$ such that
$$
{\left({\int_{z-1}^{z+1}}{{\int_0^{2T}
{{{\left|{\tilde\zeta(\tau,s)}\right|}^2}d\tau ds}}}\right)^{\frac{1}{2}}}
\leq{C_6}{e^{(\nu_1-{\varepsilon^\prime})z}}
\text{ for any } z\geq0.
$$
Noting that $\tilde\zeta$ satisfies
$$
[{g_v}(t,u,v){v_z}+({g_u}(t,u,v)-{g_u}(t,0,0)){u_z}]+{g_u}(t,0,0)\tilde\zeta
+{\tilde\zeta_{zz}}+c{\tilde\zeta_z}-{\tilde\zeta_t}=0
$$
for any $(t,z)\in\mathbb R\times{\mathbb R^+}$, where
$$
|{g_v}(t,u,v){v_z}+({g_u}(t,u,v)-{g_u}(t,0,0)){u_z}|=O(e^{\alpha z})
\text{ as } z\to+\infty,
$$
with $\alpha=\max\{\nu_2,\nu_1+\nu_1^+\}<\nu_1.$
The same argument as above shows that there exists $C_7>0$ such that
$$
\mathop{\sup}\limits_{t\in[0,T]}\left|{\tilde\zeta(t,z)}\right|
\leq {C_7}{e^{\beta z}}
\text{ for any } z\geq 0,
$$
where
$\beta=\max\{\alpha,\nu_1-\varepsilon^\prime\}
=\max\{\nu_2,\nu_1-\varepsilon^\prime\}<\nu_1.$
Therefore,
$$
\mathop{\lim}\limits_{z\to+\infty}\frac{u_z(t,z)}{{k_2}{e^{\nu_1z}}\phi_1(t)}
=\nu_1
\text{ uniformly in } t\in \mathbb R.$$
The proof is complete.
\end{proof}

Then we study the exponential behavior of $u$ as $z\to+\infty$ when $\nu_1=\nu_2.$

\begin{theorem}\label{1=2}
Assume (A1)-(A3). Let $(u(t,z),v(t,z))$ be a solution of \eqref{UV}.
If $\nu_1=\nu_2,$ then
\begin{equation}\label{u-uz-rate-1=2}
\mathop{\lim}\limits_{z\to+\infty}\frac{u(t,z)}
{\vartheta_1\left|{z}\right|{e^{\nu_1z}}\phi_1(t)}=1,
~\mathop{\lim}\limits_{z\to+\infty}\frac{u_z(t,z)}
{\vartheta_1\left|{z}\right|{e^{\nu_1z}}\phi_1(t)}=\nu_1\text{ uniformly in } t\in \mathbb R,
\end{equation}
where
$$
\vartheta_1=-(2\nu_1+c)^{-1}\overline{\varrho_1(t)},~ \varrho_1(t)=k_1b_1(t)q(t)\phi_2(t)\phi_1^{-1}(t),
$$
and $k_1$ is defined in Proposition \ref{1<2}.

\end{theorem}
\begin{proof}
Define
$$
\phi^*(t)=\int_0^t{\left[(2\nu_1+c)\vartheta_1+k_1b_1(s)q(s)\phi_2(s)\phi_1^{-1}(s)\right]}ds,
$$
then it is easy to see that $\phi^*$ is $T-$periodic in $t\in\mathbb R.$
A direct calculation shows that
$\omega(t,z):=e^{\nu_1z}\phi_1(t)(\vartheta_1\left|{z}\right|+\phi^*(t))$ satisfies
$$
g_u(t,0,0)\omega+k_1g_v(t,0,0)e^{\nu_1z}\phi_2(t)+\omega_{zz}+c\omega_z-\omega_t=0
\text{ for any } z\geq0.
$$
Let
$$
\xi(t,z)=u(t,z)-e^{\nu_1z}\phi_1(t)(\vartheta_1\left|{z}\right|+\phi^*(t)),
~\eta(t,z)=v(t,z)-k_1e^{\nu_1z}\phi_2(t),
$$
then $\xi(t,z)$ satisfies
\begin{equation*}
R(t,z)+g_v(t,0,0)\eta+g_u(t,0,0)\xi+\xi_{zz}+c\xi_z-\xi_t=0,
\end{equation*}
where
$R(t,z)=g(t,u,v)-g_u(t,0,0)u-g_v(t,0,0)v.$
It follows from \eqref{v-higer} that there exist
$K^\prime, \varepsilon^\prime, M_1>0$ such that
$$
\mathop{\sup}\limits_{t\in[0,T]}\left|{\eta (t,z)}\right|\leq{K^\prime}{e^{(\nu_2-{\varepsilon^\prime})z}}
\text{ for any } z\geq M_1.
$$
On the other hand, \eqref{general-estimate-derives-+} implies that
there exist $C_1, M_2>0$ such that
$$
\left|{R(t,z)}\right|=\left|{g(t,u,v)-g_u(t,0,0)u-g_v(t,0,0)v}\right|
\leq C_1e^{2\nu_1^+z}
\text{ for any } (t,z)\in\mathbb R\times[M_2,+\infty).
$$
Set
\begin{align*}
M=&\max\{M_1,M_2\},~\varepsilon\in\left(0,\min\{\varepsilon^\prime,\nu_1-2\nu_1^+\}\right),\\
\rho=&(\nu_1-\varepsilon)^2+c(\nu_1-\varepsilon)-\kappa_1>0,
~B=\frac{\mathop{\max}\limits_{t\in\mathbb R} g_v(t,0,0)K^\prime+C_1}{\rho\mathop{\min}\limits_{t\in\mathbb R}\phi_1(t)},\\
A\geq&\max\left\{Be^{-\varepsilon M},~
\vartheta_1M+B+\mathop{\max}\limits_{t\in\mathbb R}\phi^*(t),~
\frac{1}{e^{\nu_1M}\mathop{\min}\limits_{t\in\mathbb R}\phi_1(t)}
-\vartheta_1M-\mathop{\min}\limits_{t\in\mathbb R}\phi^*(t)+B\right\}.
\end{align*}
Moreover,
$\overline{\omega}(t,z):=(Ae^{\nu_1z}-Be^{(\nu_1-\varepsilon)z})\phi_1(t)$
satisfies
\begin{equation*}
\begin{aligned}
&\overline{\omega}_{zz}+c\overline{\omega}_z-\overline{\omega}_t+g_u(t,0,0)\overline{\omega}
+g_v(t,0,0)\eta+R(t,z)\\
&=-\rho Be^{(\nu_1-\varepsilon)z}\phi_1(t)
+g_v(t,0,0)\eta+R(t,z)\\
&\leq-\rho Be^{(\nu_1-\varepsilon)z}\phi_1(t)
+\mathop{\max}\limits_{t\in\mathbb R}g_v(t,0,0)K^\prime e^{(\nu_1-\varepsilon)z}+C_1e^{2\nu_1^+z}\\
&\leq\left(-\rho B\mathop{\min}\limits_{t\in\mathbb R}\phi_1(t)
+\mathop{\max}\limits_{t\in\mathbb R}
g_v(t,0,0)K^\prime+C_1\right)e^{(\nu_1-\varepsilon^\prime)z}\\
&\leq0 ~\text{ for any } (t,z)\in\mathbb R\times [M,+\infty).
\end{aligned}
\end{equation*}
Similarly, $\underline{\omega}(t,z):=-\overline{\omega}(t,z)
=(Be^{(\nu_1-\varepsilon)z}-Ae^{\nu_1z})\phi_1(t)$
satisfies
\begin{equation*}
\underline{\omega}_{zz}+c\underline{\omega}_z
-\underline{\omega}_t+g_u(t,0,0)\underline{\omega}
+g_v(t,0,0)\eta+R(t,z)\geq 0
\end{equation*}
for any $(t,z)\in\mathbb R\times [M,+\infty).$
Next, we prove that
\begin{equation}\label{sub-super}
(Be^{(\nu_1-\varepsilon)z}-Ae^{\nu_1z})\phi_1(t)
\leq\xi(t,z)
\leq(Ae^{\nu_1z}-Be^{(\nu_1-\varepsilon)z})\phi_1(t)
\end{equation}
for any $(t,z)\in\mathbb R\times[M,+\infty).$
Let
$$
\omega^{\pm}(t,z)=\pm(Ae^{\nu_1z}-Be^{(\nu_1-\varepsilon)z})\phi_1(t)-\xi(t,z).
$$
Since ${\omega^\pm}(t,z)$ is $T-$ periodic in $t,$
it is sufficient to show that
${\omega^+}(t,z)\geq0$ for any $(t,z)\in(0,2T)\times[M,+\infty).$
Note that
\begin{equation}\label{w+}
\omega^+_{zz}+c\omega^+_z-\omega^+_t+g_u(t,0,0)\omega^+\leq 0
\text{ for any }(t,z)\in(0,2T)\times[M,+\infty),
\end{equation}
and
\begin{equation*}
\begin{aligned}
{\omega^+}(t,M)&=(Ae^{\nu_1M}-Be^{(\nu_1-\varepsilon)M})\phi_1(t)-\left[u(t,M)
-e^{\nu_1M}\phi_1(t)(\vartheta_1 M+\phi^*(t))\right]\\
&=\left[A+\vartheta_1M+\phi^*(t)\right]e^{\nu_1M}\phi_1(t)
-Be^{(\nu_1-\varepsilon)M}\phi_1(t)-u(t,M)\\
&\geq\left[A+\vartheta_1M+\phi^*(t)-B\right]e^{\nu_1M}\phi_1(t)-1\\
&\geq 0 ~\text{ for any } t\in(0,2T).
\end{aligned}
\end{equation*}
Assume on the contrary that
$$
\mathop{\inf}\limits_{(t,z)\in(0,2T)\times[M,+\infty)}{\omega^+}(t,z)<0,
$$
due to
$\mathop{\lim}\limits_{z\to+\infty}\mathop{\sup}\limits_{t\in(0,2T)}{\omega^+}(t,z)=0,$
then there exists $({t^*},{z^*})\in(0,2T)\times(M,+\infty )$ such that
$$
{\omega^+}({t^*},{z^*})=\mathop {\inf}\limits_{(t,z)\in(0,2T)\times [M,+\infty)}{\omega^+}(t,z)<0,
$$
and hence
$\left[{\omega_{zz}^++c\omega_z^+
-\omega_t^++g_u(t,0,0){\omega^+}}\right]\big|_{({t^*},{z^*})}>0,$
which contradicts to \eqref{w+}.
We can show $\omega^-(t,z)\leq 0$ by a similar argument.
It then follows from \eqref{sub-super} that
$$
\mathop{\sup}\limits_{t\in\mathbb R}\left|{\xi(t,z)}\right|=o({e^{{\nu_1}z}})
\text{ as } z\to+\infty.
$$
Therefore, by the definition of $\xi(t,z)$, we have
$$
\mathop{\lim}\limits_{z\to+\infty}\frac{u(t,z)}
{{\vartheta_1}\left|{z}\right|{e^{\nu_1z}}\phi_1(t)}=1
\text{ uniformly in } t\in \mathbb R.
$$

The argument for $u_z$ is similar. Let
$$
\tilde\xi(t,z)=u_z(t,z)
-{e^{\nu_1z}}\phi_1(t)\left[\vartheta_1(\nu_1\left|{z}\right|+1)+\nu_1\phi^*(t)\right],~
\tilde\eta(t,z)=v_z(t,z)-\nu_1k_1{e^{\nu_1z}}\phi_2(t).
$$
Then
\begin{equation*}
\tilde R(t,z)+g_v(t,0,0)\tilde\eta+g_u(t,0,0)\tilde\xi
+\tilde\xi_{zz}+c\tilde\xi_z-\tilde\xi_t=0
\text{ for any } (t,z)\in\mathbb R\times[0,+\infty),
\end{equation*}
where
$\tilde R(t,z)
=\left[{g_u(t,u,v)-g_u(t,0,0)}\right]u_z+\left[{g_v(t,u,v)-g_v(t,0,0)}\right]v_z.$
Since \eqref{v-higer} shows that
$$
\mathop{\sup}\limits_{t\in[0,T]}\left|{\tilde\eta (t,z)}\right|\leq{K^{\prime\prime}}{e^{(\nu_2-{\varepsilon^\prime})z}}
\text{ for any } z\geq 0,
$$
and \eqref{general-estimate-derives-+} implies that there exists $C_2>0$ such that
$$
\left|{\tilde R(t,z)}\right|\leq C_2e^{2\nu_1^+z} \text{ for any } (t,z)\in\mathbb R\times[0,+\infty),
$$
a similar argument as above yields that
$\mathop{\sup}\limits_{t\in\mathbb R}\left|{\tilde\xi(t,z)}\right|=o({e^{\nu_1z}})$
as $z\to+\infty,$
and hence
$$
\mathop{\lim}\limits_{z\to+\infty}\frac{u_z{(t,z)}}
{{\vartheta_1}\left|{z}\right|{e^{\nu_1z}}\phi_1(t)}=\nu_1
\text{ uniformly in } t\in\mathbb R.
$$
The proof is complete.
\end{proof}

\begin{proof}[\textbf{Proof of Theorem \ref{+infinity}}]
Based on the equivalence of systems \eqref{UV} and \eqref{PQ}, all the conclusions in Theorem \ref{+infinity} can be easily verified with the help of Proposition \ref{1<2},
Theorems \ref{1>2} and \ref{1=2}.
\end{proof}

\begin{proof}[\textbf{Proof of Theorem \ref{-infinity}}]
As $z\to-\infty$, by a parallel discussion to that for $z\to \infty$ in Proposition \ref{1<2},
Theorems \ref{1>2} and \ref{1=2}, we can prove
\begin{equation*}
\begin{aligned}
&\mathop{\lim}\limits_{z\to-\infty}\frac{P(t,z)}{{k_3}{e^{\nu_3z}}\psi_1(t)}=1,
~\mathop{\lim}\limits_{z\to-\infty}\frac{{P_z}(t,z)}{{k_3}{e^{\nu_3z}}\psi_1(t)}=\nu_3
\text{ uniformly in } t\in \mathbb R.\\
&\mathop{\lim}\limits_{z\to-\infty}\frac{Q(t,z)}{{k_3}{e^{\nu_3z}}\tilde{\psi}_2(t)}=1,
~\mathop{\lim}\limits_{z\to-\infty}\frac{{Q_z}(t,z)}{{k_3}{e^{\nu_3z}}\tilde{\psi}_2(t)}=\nu_3
\text{ uniformly in } t\in \mathbb R, \text{ if } \nu_4>\nu_3,\\
&\mathop{\lim}\limits_{z\to-\infty}\frac{Q(t,z)}{{\vartheta_2|z|}{e^{\nu_4z}}\psi_2(t)}=1,
~\mathop{\lim}\limits_{z\to-\infty}\frac{{Q_z}(t,z)}{{\vartheta_2|z|}{e^{\nu_4z}}\psi_2(t)}=\nu_4
\text{ uniformly in } t\in \mathbb R, \text{ if } \nu_4=\nu_3,\\
&\mathop{\lim}\limits_{z\to-\infty}\frac{Q(t,z)}{{k_4}{e^{\nu_4z}}\psi_2(t)}=1,
~\mathop{\lim}\limits_{z\to-\infty}\frac{{Q_z}(t,z)}{{k_4}{e^{\nu_4z}}\psi_2(t)}=\nu_4
\text{ uniformly in } t\in \mathbb R, \text{ if } \nu_4<\nu_3,
\end{aligned}
\end{equation*}
where $k_i>0~(i=3,4)$ are constants,
$$
\psi_1(t)=e^{\int_0^t{(b_1(s)q(s)-a_1(s)p(s))ds}+\kappa_3t},
~\psi_2(t)=e^{\int_0^t{b_2(s)q(s)ds}+\kappa_4t},
$$
$$
\vartheta_2=-(2\nu_4+c)^{-1}\overline{\varrho_2(t)},~ \varrho_2(t)=k_3a_2(t)p(t)\psi_1(t)\psi_2^{-1}(t),
$$
and
\begin{equation*}
\begin{cases}
\tilde{\psi}_2(t)=\tilde{\psi}_2(0){e^{\int_0^t{(\upsilon_2-b_2(s)q(s))ds}}}
+\int_0^t{{e^{\int_s^t{(\upsilon_2-b_2(\tau)q(\tau))d\tau}}}a_2(s)p(s)\psi_1(s)ds},\\
\tilde{\psi}_2(0)={\left({1-{e^{\int_0^T{(\upsilon_2-b_2(s)q(s))ds}}}}\right)^{-1}}
\int_0^T{{e^{\int_s^T{(\upsilon_2-b_2(\tau)q(\tau))d\tau}}}a_2(s)p(s)\psi_1(s)ds}
\end{cases}
\end{equation*}
with $\upsilon_2=\nu_3^2+c\nu_3$. Because of the limited passage, we omit the details here.
\end{proof}

\section{\bf A pair of sub- and supersolutions}

In this section, we establish a comparison theorem and
construct a pair of sub- and supersolutions for system \eqref{uv-trans-eq}.
Let
\begin{equation*}
\begin{cases}
\mathcal{F}_1(t,u,v)=u_t-u_{xx}-f(t,u,v),\\
\mathcal{F}_2(t,u,v)=v_t-dv_{xx}-l(t,u,v),
\end{cases}
\end{equation*}
where
$f(t,u,v)=a_1pu[1-u-N_1(t)(1-v)]$ and $l(t,u,v)=b_2q(1-v)[N_2(t)u-v].$
Then \eqref{uv-trans-eq} can be written as
\begin{equation}\label{F1F2}
\begin{cases}
\mathcal{F}_1(t,u,v)=0,\\
\mathcal{F}_2(t,u,v)=0.
\end{cases}
\end{equation}

\begin{definition}\label{super-sub-def}
A pair of continuous functions $\bm{w}(t,x)=(u(t,x),v(t,x))$
is said to be a supersolution~(subsolution)
of \eqref{F1F2} in $(t,x)\in\mathbb R^+\times\mathbb R,$
if the required derivatives exist and there hold
\begin{equation*}
\begin{cases}
\mathcal{F}_1(t,u,v)=u_t-u_{xx}-f(t,u,v)\geq 0 ~(\leq0),~~(t,x)\in\mathbb R^+\times\mathbb R,\\
\mathcal{F}_2(t,u,v)=v_t-dv_{xx}-l(t,u,v)\geq 0 ~(\leq0),~~(t,x)\in\mathbb R^+\times\mathbb R.
\end{cases}
\end{equation*}
\end{definition}

By similar arguments as in \cite[Theorem 2.2 and Corollary 2.3]{wang2012},
we first state the following comparison principle for cooperative system \eqref{F1F2}.

\begin{lemma}\label{sub-super method}
\rm(i) Suppose that
$\bm{w}^+(t,x)$ and $\bm{w}^-(t,x)$
are super- and subsolutions of \eqref{F1F2} in
$(t,x)\in\mathbb R^+\times\mathbb R,$ respectively,
$\bm{0}\leq\bm{w}^\pm(t,x)\leq\bm{1}.$
If
$\bm{w}^-(0,x)\leq\bm{w}^+(0,x)$ for any $x\in\mathbb R,$
then
$\bm{w}^-(t,x)\leq\bm{w}^+(t,x)$
for any
$(t,x)\in\mathbb R^+\times\mathbb R.$

\rm(ii) For any $\bm{0}\leq\bm{w}_0\leq\bm{1}$ satisfying
$\bm{w}^-(0,\cdot)\leq\bm{w}_0(\cdot)\leq\bm{w}^+(0,\cdot)$ in $x\in\mathbb R,$
there hold
$\bm{w}^-(t,x)\leq\bm{w}(t,x;\bm{w}_0)\leq\bm{w}^+(t,x)$
and
$\bm{0}\leq\bm{w}(t,x;\bm{w}_0)\leq\bm{1}$
for any $(t,x)\in\mathbb R^+\times\mathbb R,$
where
$\bm{w}(t,x;\bm{w}_0)$ is the unique classical solution of
\eqref{F1F2} with initial value
$\bm{w}(0,x;\bm{w}_0)=\bm{w}_0.$
\end{lemma}

The subsequent lemma is a straightforward consequence of Theorems \ref{+infinity} and \ref{-infinity}.

\begin{lemma}\label{basic-estimate}
Let $(P(t,z),Q(t,z))$ be a traveling wave solution of \eqref{uv-trans-eq}
satisfying (C1).
Then there exist positive constants
$M, N, M_1, m_1, \delta_i$ and $\gamma_i~(i=1,2)$
such that
\begin{gather}
Q(t,z)\leq M P(t,z),
\quad t\in \mathbb R ,~z\leq 0 ,\label{QP-}\\
m_1e^{\nu_3z}\leq P(t,z)\leq M_1e^{\nu_3z},
\quad t\in \mathbb R ,~z \leq 0 ,\label{P-}\\
\delta_1P(t,z)\leq P_z(t,z)\leq\delta_2P(t,z),
\quad t\in \mathbb R ,~z \leq 0 ,\label{PPz-}\\
\gamma_1Q(t,z)\leq Q_z(t,z)\leq\gamma_2Q(t,z),
\quad t\in \mathbb R ,~z \leq 0 ,\label{QQz-}\\
1-Q(t,z)\leq N(1-P(t,z)),
\quad t\in \mathbb R ,~z \geq 0 ,\label{QP+}\\
\delta_1(1-P(t,z))\leq P_z(t,z),
\quad t\in \mathbb R ,~z \geq 0 ,\label{PPz+}\\
\gamma_1(1-Q(t,z))\leq Q_z(t,z),
\quad t\in \mathbb R ,~z \geq 0.\label{QQz+}
\end{gather}
\end{lemma}

The next two lemmas give some key estimates which we need to construct appreciate sub- and supersolutions later.

\begin{lemma}\label{H1-estimate}
Let
$(P(t,z),Q(t,z))$ be a solution of \eqref{PQ} satisfying (C1).
Denote
$$P_1=P(t,x+j_1), P_2=P(t,-x+j_2), Q_1=Q(t,x+j_1), Q_2=Q(t,-x+j_2),$$
$$H(t,x)=2P_{1,z}P_{2,z}+b_1q[P_1Q_2(1-P_2)(1-Q_1)+P_2Q_1(1-P_1)(1-Q_2)],$$
where $j_2\leq j_1\leq 0.$
Then there exists some $K_1>0$ such that
\begin{equation}\label{H1}
\frac{H(t,x)}{(1-P_2)P_{1,z}+(1-P_1)P_{2,z}}\leq K_1e^{\nu_3j_1}
\text{ for any } (t,x)\in \mathbb{R}\times\mathbb{R}.
\end{equation}
\end{lemma}
\begin{proof}
We divide $x\in \mathbb R$ into four intervals.
\begin{itemize}
\setlength{\itemindent}{1.4em}
\item[Case A:] $j_2\leq x\leq 0.$ Then $x+j_1\leq 0$ and $-x+j_2\leq 0.$
By \eqref{QP-}, \eqref{P-} and \eqref{PPz-},
\begin{align*}
&\frac{H(t,x)}{(1-P_2)P_{1,z}+(1-P_1)P_{2,z}}
\leq\frac{2P_{1,z}P_{2,z}+b_1q(P_1Q_2+P_2Q_1)}{(1-P_1)P_{2,z}}\\
&\leq\frac{2P_{1,z}}{1-P_1}
+b_1q\left[\frac{P_1MP_2}{(1-P_1)\delta_1P_2}+\frac{P_2MP_1}{(1-P_1)\delta_1P_2}\right]\\
&\leq\frac{2\delta_2M_1e^{\nu_3(x+j_1)}}{1-P_1(t,0)}
+b_1q\left[\frac{MM_1e^{\nu_3(x+j_1)}}{(1-P_1(t,0))\delta_1}
+\frac{MM_1e^{\nu_3(x+j_1)}}{(1-P_1(t,0))\delta_1}\right]\\
&\leq\left[\frac{2\delta_2M_1}{1-P_1(t,0)}+
\mathop{\max}\limits_{t\in[0,T]}(b_1q)
\frac{2MM_1}{(1-P_1(t,0))\delta_1}\right]e^{\nu_3j_1},
~t\in\mathbb R.
\end{align*}
\item[Case B:] $0\leq x\leq -j_1.$
Then $x+j_1\leq 0$ and $-x+j_2\leq0.$
Similar to Case A, there holds
\begin{align*}
\hspace{-2em}
\frac{H(t,x)}{(1-P_2)P_{1,z}+(1-P_1)P_{2,z}}
\leq\left[\frac{2\delta_2M_1}{1-P_2(t,0)}+
\mathop{\max}\limits_{t\in[0,T]}(b_1q)
\frac{2MM_1}{(1-P_2(t,0))\delta_1}\right]e^{\nu_3j_1},
~t\in\mathbb R.
\end{align*}
\item[Case C:] $x\geq -j_1.$
Then $x+j_1\geq 0$ and $-x+j_2\leq 0.$
By \eqref{QP-}, \eqref{P-}, \eqref{PPz-}, \eqref{QP+} and \eqref{PPz+},
\hspace{-1.2em}
\begin{align*}
&\frac{H(t,x)}{(1-P_2)P_{1,z}+(1-P_1)P_{2,z}}
\leq\frac{2P_{1,z}P_{2,z}+b_1q[Q_2(1-Q_1)+P_2(1-P_1)]}{(1-P_2)P_{1,z}}\\
&\leq\frac{2\delta_2M_1e^{\nu_3(-x+j_2)}}{1-P_2(t,0)}
+b_1q\left[\frac{MM_1e^{\nu_3(-x+j_2)}N(1-P_1)}{(1-P_2(t,0))\delta_1(1-P_1)}
+\frac{M_1e^{\nu_3(-x+j_2)}(1-P_1)}{(1-P_2(t,0))\delta_1(1-P_1)}\right] \\
&\leq\left[\frac{2\delta_2M_1}{1-P_2(t,0)}+\mathop{\max}\limits_{t\in[0,T]}(b_1q)
\frac{M_1(MN+1)}{(1-P_2(t,0))\delta_1}\right]e^{\nu_3j_1},
~t\in\mathbb R.
\end{align*}
\item[Case D:] $x \leq j_2.$
Then $x+j_1\leq 0$ and $-x+j_2\geq 0.$
Similar to Case C,
\begin{equation*}
\hspace{-2.5em}
\frac{H(t,x)}{(1-P_2)P_{1,z}+(1-P_1)P_{2,z}}
\leq\left[\frac{2\delta_2M_1}{1-P_1(t,0)}+\mathop{\max}\limits_{t\in[0,T]}(b_1q)
\frac{M_1(MN+1)}{(1-P_1(t,0))\delta_1}\right]e^{\nu_3j_1},
~t\in \mathbb R.
\end{equation*}
\end{itemize}

Let
$$
K_1=M_1\mathop{\max}\limits_{\begin{subarray}{l}i=1,2\\t\in[0,T]\end{subarray}}
\left\{{\frac{2\delta_2}{1-P_i(t,0)}+\frac{2C_2^+M}{(1-P_i(t,0))\delta_1},
\frac{2\delta_2}{1-P_i(t,0)}+\frac{C_2^+(MN+1)}{(1-P_i(t,0))\delta_1}}\right\},
$$
where $C_2^+=\mathop{\max}\limits_{t\in[0,T]}b_1(t)q(t).$
Then \eqref{H1} holds and we complete the proof.
\end{proof}

\begin{lemma}\label{H2-estimate}
Let $(P(t,z),Q(t,z))$ be a solution of \eqref{PQ} satisfying (C1).
Denote
$$
Q_1=Q(t,x+j_1), Q_2=Q(t,-x+j_2), \tilde{H}(t,x)=2dQ_{1,z}Q_{2,z}+b_2qQ_1Q_2(1-Q_1)(1-Q_2),
$$
where $j_2\leq j_1\leq 0.$
Then there exists some $K_2>0$ such that
\begin{equation}\label{H2}
\frac{\tilde{H}(t,x)}{(1-Q_2)Q_{1,z}+(1-Q_1)Q_{2,z}}\leq K_2e^{\nu_3j_1}
\text{ for any } (t,x)\in \mathbb{R}\times\mathbb{R}.
\end{equation}
\end{lemma}

\begin{proof}
We divide $x\in \mathbb R$ into four intervals.
\begin{itemize}
\setlength{\itemindent}{1.3em}
\item[Case A:] $j_2\leq x\leq 0.$
Then $x+j_1\leq 0$ and $-x+j_2\leq 0.$
By \eqref{QP-}, \eqref{P-} and \eqref{QQz-},
\hspace{-2em}
\begin{align*}
&\frac{\tilde{H}(t,x)}{(1-Q_2)Q_{1,z}+(1-Q_1)Q_{2,z}}
\leq\frac{2dQ_{1,z}}{1-Q_1}+\frac{b_2qQ_1Q_2}{Q_{2,z}}\\
&\leq\frac{2d\gamma_2MM_1e^{\nu_3(x+j_1)}}{1-Q_1(t,0)}
+b_2q\frac{MM_1e^{\nu_3(x+j_1)}Q_2}{\gamma_1Q_2}\\
&\leq\left[\frac{2d\gamma_2MM_1}{1-Q_1(t,0)}
+\mathop{\max}\limits_{t\in[0,T]}(b_2q)\frac{MM_1}{\gamma_1}\right]e^{\nu_3j_1},
~t\in \mathbb R.
\end{align*}
\item[Case B:] $0\leq x\leq -j_1.$
Then $x+j_1\leq 0$ and $-x+j_2\leq 0.$
Similar to Case A,
\hspace{-3em}
\begin{align*}
\frac{\tilde{H}(t,x)}{(1-Q_2)Q_{1,z}+(1-Q_1)Q_{2,z}}
\leq\left({\frac{2d\gamma_2MM_1}{1-Q_2(t,0)}
+\mathop{\max}\limits_{t\in[0,T]}(b_2q)\frac{MM_1}{\gamma_1}}\right)e^{\nu_3j_1},
~t\in \mathbb R.
\end{align*}
\item[Case C:] $x\geq -j_1.$
Then $x+j_1\geq 0$ and $-x+j_2\leq 0.$
By \eqref{QP-}, \eqref{P-}, \eqref{QQz-} and \eqref{QQz+},
\begin{align*}
&\frac{\tilde{H}(t,x)}{(1-Q_2)Q_{1,z}+(1-Q_1)Q_{2,z}}
\leq\frac{2dQ_{2,z}}{1-Q_2}+b_2q\frac{Q_2(1-Q_1)}{Q_{1,z}}\\
&\leq\frac{2d\gamma_2MM_1e^{\nu_3(-x+j_2)}}{1-Q_2(t,0)}
+b_2q\frac{MM_1e^{\nu_3(-x+j_2)}(1-Q_1)}{\gamma_1(1-Q_1)}\\
&\leq\left[{\frac{2d\gamma_2MM_1}{1-Q_2(t,0)}
+\mathop{\max}\limits_{t\in[0,T]}(b_2q)\frac{MM_1}{\gamma_1}}\right]e^{\nu_3j_1},
~t\in \mathbb R.
\end{align*}
\item[Case D:] $x\leq j_2.$
Then $x+j_1\leq 0$ and $-x+j_2\geq 0.$
Similar to Case C,
\begin{equation*}
\hspace{-4em}
\frac{\tilde{\tilde{H}}(t,x)}{(1-Q_2)Q_{1,z}+(1-Q_1)Q_{2,z}}
\leq\left[{\frac{2d\gamma_2MM_1}{1-Q_1(t,0)}
+\mathop{\max}\limits_{t\in[0,T]}(b_2q)\frac{MM_1}{\gamma_1}}\right]e^{\nu_3j_1},
~t\in \mathbb R.
\end{equation*}
\end{itemize}

Let
$K_2=MM_1\mathop{\max}\limits_{\begin{subarray}{l}i=1,2\\t\in [0,T]\end{subarray}}
\left\{{\frac{2d\gamma_2}{1-Q_i(t,0)}+\frac{C_1^-}{\gamma_1}}\right\}$
with $C_1^-=\mathop{\max}\limits_{t\in[0,T]}b_2(t)q(t).$
Then \eqref{H2} holds and the proof is complete.
\end{proof}

In order to construct a supersolution of \eqref{F1F2},
we first introduce two functions $j_1(t)$ and $j_2(t).$
Let $K=\max\{K_1,K_2\},$ where $K_1$ and $K_2$ are defined
in Lemmas \ref{H1-estimate} and \ref{H2-estimate}, respectively.
Suppose that $c<0,$ for any $\varrho_1\in(-\infty,0],$ denote
$$
\omega_1(\varrho_1)=\varrho_1-\frac{1}{\nu_3}
\ln\left({1-\frac{K}{c}e^{\nu_3\varrho_1}}\right),
~\varpi=\omega_1(0)<0.
$$
Since $\omega_1(\varrho_1)$ is increasing in $\varrho_1\in(-\infty,0],$ we denote
its inverse function by $\varrho_1=\varrho_1(\omega_1):(-\infty,\varpi]\to(-\infty,0].$
For any $(\omega_1,\omega_2)\in(-\infty,\varpi]^2,$ define
$$
\varrho_2(\omega_1,\omega_2)=\omega_2+\frac{1}{\nu_3}
\ln\left({1-\frac{K}{c}e^{\nu_3 \varrho_1(\omega_1)}}\right).
$$
Consider the following system
\begin{equation}\label{j}
\begin{cases}
\tilde{j_1}^\prime(t;\omega_1)=-c+Ke^{\nu_3\tilde{j_1}(t;\omega_1)},~~t<0,\\
\tilde{j_2}^\prime(t;\omega_1,\omega_2)=-c+Ke^{\nu_3\tilde{j_1}(t;\omega_1)},~~t<0,\\
\tilde{j_1}(0;\omega_1)=\varrho_1(\omega_1),\\
\tilde{j_2}(0;\omega_1,\omega_2)=\varrho_2(\omega_1,\omega_2).
\end{cases}
\end{equation}
Solving \eqref{j} explicitly, we have
\begin{equation*}
\begin{cases}
\tilde{j_1}(t;\omega_1)=\varrho_1(\omega_1)-ct-\frac{1}{\nu_3}\ln\left\{{1-\frac{K}{c}e^{\nu_3 \varrho_1(\omega_1)}(1-e^{-c\nu_3 t})}\right\},\quad t\leq0,\\
\tilde{j_2}(t;\omega_1,\omega_2)=\varrho_2(\omega_1,\omega_2)
-ct-\frac{1}{\nu_3}\ln\left\{{1-\frac{K}{c}e^{\nu_3 \varrho_1(\omega_1)}(1-e^{-c\nu_3 t})}\right\},~t\leq0.
\end{cases}
\end{equation*}
It is easy to see that
$$
\tilde{j_2}(t;\omega_1,\omega_2)-\tilde{j_1}(t;\omega_1)
=\varrho_2(\omega_1,\omega_2)-\varrho_1(\omega_1)=\omega_2-\omega_1,
$$
$$
\tilde{j_1}(t;\omega_1)-(-ct+\omega_1)=
\tilde{j_2}(t;\omega_1,\omega_2)-(-ct+\omega_2)=
-\frac{1}{\nu_3}\ln\left({1-\frac{\varsigma}{1+\varsigma}e^{-c\nu_3 t}}\right)
$$
for $t\leq0,$ where
$\varsigma=-\frac{K}{c}e^{\nu_3\varrho_1(\omega_1)}.$
Moreover,
for any $(\omega_1,\omega_2)\in(-\infty,\varpi]^2,$
let
$$j_1(t)=j_1(t;\omega_1,\omega_2):=\tilde{j_1}(t;\omega_1),
~~j_2(t)=j_2(t;\omega_1,\omega_2):=\tilde{j_2}(t;\omega_1,\omega_2),
\text{ if }\omega_2\leq\omega_1;$$
$$j_1(t)=j_1(t;\omega_1,\omega_2):=\tilde{j_2}(t;\omega_1,\omega_2),
~~j_2(t)=j_2(t;\omega_1,\omega_2):=\tilde{j_1}(t;\omega_1),
\text{ if }\omega_1\leq\omega_2.$$
Then $j_2(t)\leq j_1(t)\leq0$ if $\omega_2\leq\omega_1,$ and
$j_1(t)\leq j_2(t)\leq0$ if $\omega_1\leq\omega_2.$
Therefore, there is a constant $R_0>0$ such that
\begin{equation}
0<j_1(t)-(-ct+\omega_1)=j_2(t)-(-ct+\omega_2)\leq R_0e^{-c\nu_3t},~t\leq0. \label{100}
\end{equation}

\begin{lemma}\label{super-solu}
Let $\bm\Phi(t,z)=(P(t,z),Q(t,z))$ be a traveling wave solution of \eqref{uv-trans-eq}
satisfying (C1) with $c<0.$
Then for any $(\omega_1,\omega_2)\in(-\infty,\varpi]^2,$
the pair $\bm{\overline{W}}=(\overline{U},\overline{V})$ defined by
\begin{equation*}
\bm{\overline{W}}(t,x):=\bm\Phi(t,x+j_1(t))+\bm\Phi(t,-x+j_2(t))
-\bm\Phi(t,x+j_1(t))\bm\Phi(t,-x+j_2(t))
\end{equation*}
is a supersolution of \eqref{F1F2} in
$(t,x)\in(-\infty,0]\times\mathbb R.$
\end{lemma}

\begin{proof}
Without loss of generality, we suppose that $\omega_2\leq\omega_1,$
then $j_2(t)\leq j_1(t)\leq 0$ and
$j_i^\prime(t)=-c+Ke^{\nu_3j_1(t)}~(i=1,2)$ for $t\leq 0.$
Denote
$$
P_1=P_1(t,z)=P(t,x+j_1(t)) \text{ and } P_2=P_2(t,z)=P(t,-x+j_2(t)),
$$
then a direct calculation gives
\begin{equation*}
\begin{aligned}
\mathcal{F}_1(t,\overline U,\overline V)
=&(P_{1,t}-cP_{1,z}-P_{1,zz})(1-P_2)+(P_{2,t}-cP_{2,z}-P_{2,zz})(1-P_1)\\
&+Ke^{\nu_3 j_1(t)}\left[(1-P_2)P_{1,z}+(1-P_1)P_{2,z}\right]-2P_{1,z}P_{2,z}\\
&-f(t,P_1+P_2-P_1P_2,Q_1+Q_2-Q_1Q_2)\\
=&Ke^{\nu_3 j_1(t)}\left[(1-P_2)P_{1,z}+(1-P_1)P_{2,z}\right]-2P_{1,z}P_{2,z}\\
&+(1-P_2)f(t,P_1,Q_1)+(1-P_1)f(t,P_2,Q_2)\\
&-f(t,P_1+P_2-P_1P_2,Q_1+Q_2-Q_1Q_2)\\
:=&A(t,x)Ke^{\nu_3 j_1(t)}-H_1(t,x),
\end{aligned}
\end{equation*}
where
$A(t,x)=(1-P_2)P_{1,z}+(1-P_1)P_{2,z}>0$
and
\begin{equation*}
\begin{aligned}
H_1(t,x)=&2P_{1,z}P_{2,z}+f(t,P_1+P_2-P_1P_2,Q_1+Q_2-Q_1Q_2)\\
&-(1-P_2)f(t,P_1,Q_1)-(1-P_1)f(t,P_2,Q_2).
\end{aligned}
\end{equation*}
Note that
\begin{equation*}
\begin{aligned}
&f(t,P_1+P_2-P_1P_2,Q_1+Q_2-Q_1Q_2)
-(1-P_2)f(t,P_1,Q_1)-(1-P_1)f(t,P_2,Q_2)\\
&=a_1p(P_1+P_2-P_1P_2)\left[(1-P_1)(1-P_2)-N_1(1-Q_1)(1-Q_2)\right]\\
&\quad-a_1p\left[P_1(1-P_2)(1-P_1-N_1(1-Q_1))+P_2(1-P_1)(1-P_2-N_1(1-Q_2))\right]\\
&=a_1p\left[-P_1P_2(1-P_1)(1-P_2)+N_1P_1Q_2(1-P_2)(1-Q_1)\right .\\
&\quad+\left. N_1P_2Q_1(1-P_1)(1-Q_2)-N_1P_1P_2(1-Q_1)(1-Q_2) \right]\\
&\leq b_1q\left[P_1Q_2(1-P_2)(1-Q_1)+P_2Q_1(1-P_1)(1-Q_2)\right],
\end{aligned}
\end{equation*}
it follows that $H_1(t,x)\leq H(t,x).$
By Lemma \ref{H1-estimate}, there hold
$$
\frac{H_1(t,x)}{A(t,x)}\leq \frac{H(t,x)}{A(t,x)}\leq Ke^{\nu_3 j_1(t)}
\text{ for any } (t,x)\in(-\infty,0]\times\mathbb{R},
$$
and hence
$\mathcal{F}_1(t,\overline U,\overline V)\geq 0$
for any $(t,x)\in (-\infty,0]\times \mathbb{R}.$

We now prove $\mathcal{F}_2(t,\overline U,\overline V)\geq 0.$ Note that
\begin{equation*}
\begin{aligned}
\mathcal{F}_2(t,\overline U,\overline V)
=&(Q_{1,t}-cQ_{1,z}-dQ_{1,zz})(1-Q_2)+(Q_{2,t}-cQ_{2,z}-dQ_{2,zz})(1-Q_1)\\
&-2dQ_{1,z}Q_{2,z}+Ke^{\nu_3 j_1(t)}\left[(1-Q_2)Q_{1,z}+(1-Q_1)Q_{2,z}\right]\\
&-l(t,P_1+P_2-P_1P_2,Q_1+Q_2-Q_1Q_2)\\
=&Ke^{\nu_3 j_1(t)}[(1-Q_2)Q_{1,z}+(1-Q_1)Q_{2,z}]-2dQ_{1,z}Q_{2,z}+(1-Q_1)l(t,P_2,Q_2)\\
&+(1-Q_2)l(t,P_1,Q_1)-l(t,P_1+P_2-P_1P_2,Q_1+Q_2-Q_1Q_2)\\
:=&B(t,x)Ke^{\nu_3 j_1(t)}-H_2(t,x),
\end{aligned}
\end{equation*}
where $B(t,x)=(1-Q_2)Q_{1,z}+(1-Q_1)Q_{2,z}>0$ and
\begin{equation*}
\begin{aligned}
H_2(t,x)=&2dQ_{1,z}Q_{2,z}+l(t,P_1+P_2-P_1P_2,Q_1+Q_2-Q_1Q_2)\\
&-(1-Q_1)l(t,P_2,Q_2)-(1-Q_2)l(t,P_1,Q_1).
\end{aligned}
\end{equation*}
Since
\begin{equation*}
\begin{aligned}
&l(t,P_1+P_2-P_1P_2,Q_1+Q_2-Q_1Q_2)-(1-Q_1)l(t,P_2,Q_2)-(1-Q_2)l(t,P_1,Q_1)\\
&=b_2q\left[(1-Q_1)(1-Q_2)\left[N_2(P_1+P_2-P_1P_2)-(Q_1+Q_2-Q_1Q_2)\right]\right]\\
&\quad-b_2q\left[(1-Q_1)(1-Q_2)(N_2P_1-Q_1)+(1-Q_1)(1-Q_2)(N_2P_2-Q_2)\right]\\
&=b_2q\left[(1-Q_1)(1-Q_2)(Q_1Q_2-N_2P_1P_2)\right]\\
&\leq b_2qQ_1Q_2(1-Q_1)(1-Q_2),
\end{aligned}
\end{equation*}
it follows that $H_2(t,x)\leq \tilde{H}(t,x).$
By Lemma \ref{H2-estimate}, there hold
$$
\frac{H_2(t,x)}{B(t,x)}\leq \frac{\tilde{H}(t,x)}{B(t,x)}\leq Ke^{\nu_3 j_1(t)}
\text{ for any }(t,x)\in (-\infty,0]\times \mathbb{R}.
$$
Therefore,
$\mathcal{F}_2(t,\overline U,\overline V)\geq 0$
for any $(t,x)\in (-\infty,0]\times\mathbb{R}.$ This ends the proof.
\end{proof}

A subsolution of \eqref{F1F2} can be easily obtained as follows.

\begin{lemma}\label{subsolution}
Let $\bm\Phi(t,z)=(P(t,z),Q(t,z))$ be a traveling wave solution of \eqref{uv-trans-eq}
satisfying (C1) with $c<0.$
Then for any $(\omega_1,\omega_2)\in(-\infty,\varpi]^2,$
the pair $\bm{\underline{w}}=(\underline{u},\underline{v})$ defined by
\begin{equation*}
\bm{\underline{w}}(t,x)=\bm{\underline{w}}(t,x;\omega_1,\omega_2)
=\mathop{\max}\left\{\bm\Phi(t,x-ct+\omega_1),\bm\Phi(t,-x-ct+\omega_2)\right\}
\end{equation*}
is a subsolution of \eqref{F1F2} in $(t,x)\in\mathbb R\times\mathbb R$ in the sense of distribution.
\end{lemma}

\begin{remark}\label{su-su}\rm
We can easily see that $\bm{0}<\bm{\underline{w}}(t,x)
\leq\bm{\overline{W}}(t,x)<\bm{1}$ for any $(t,x)\in\mathbb R\times\mathbb R.$
\end{remark}

\section{\bf Entire solutions }

In this section, we establish the existence and some qualitative properties of
entire solutions using the comparison argument coupled with super- and subsolutions method.

\begin{theorem}\label{existence}
Let $\bm\Phi(t,z)=(P(t,z),Q(t,z))$ be a traveling wave solution of \eqref{uv-trans-eq}
satisfying (C1) with $c<0.$
Then for any $(\omega_1,\omega_2)\in(-\infty,\varpi]^2,$ system \eqref{uv-trans-eq}
admits an entire solution
$\bm{W}_{\omega_1,\omega_2}(t,x)
=(U_{\omega_1,\omega_2}(t,x),V_{\omega_1,\omega_2}(t,x))$
satisfying $\bm{0}<\bm{W}_{\omega_1,\omega_2}(t,x)<\bm{1}$. Moreover, the following statements are valid:
\begin{description}
\item[(i)]
$\bm{W}_{\omega_1,\omega_2}(t+T,x)=\bm{W}_{\omega_1,\omega_2}(t,x)$
or
$\bm{W}_{\omega_1,\omega_2}(t+T,x)>\bm{W}_{\omega_1,\omega_2}(t,x)$
for any $(t,x)\in\mathbb R^2.$

\item[(ii)]
\begin{equation*}
\begin{aligned}
\lim_{t\to-\infty}
&\left\{\sup_{x\geq 0}
{\left|\bm{W}_{\omega_1,\omega_2}(t,x)-\bm\Phi(t,x-ct+\omega_1)\right|}\right .\\
&+\left .\sup_{x\leq0}
{\left|\bm{W}_{\omega_1,\omega_2}(t,x)-\bm\Phi(t,-x-ct+\omega_2)\right|}
\right\}
=0.
\end{aligned}
\end{equation*}

\item[(iii)]
$\mathop{\lim}\limits_{k\to+\infty}\mathop{\sup}\limits_{(t,x)\in[0,T]\times\mathbb R}
\left|\bm{W}_{\omega_1,\omega_2}(t+kT,x)-\bm{1}
\right|=0.$

\item[(iv)]
$\mathop{\lim}\limits_{k\to-\infty}\mathop{\sup}\limits_{(t,x)\in[0,T]\times[a,b]}
\left|\bm{W}_{\omega_1,\omega_2}(t+kT,x)\right|=0$ for any $a,b\in\mathbb{R}$ with $a<b$.

\item[(v)] $\mathop{\lim}\limits_{|x|\to+\infty}\mathop{\sup}\limits_{t\in[t_0,+\infty)}
\left|\bm{W}_{\omega_1,\omega_2}(t,x)-\bm{1}\right|=0$ for any $t_0\in\mathbb R.$

\item[(vi)]
$\bm{W}_{\omega_1,\omega_2}(t,x)$ is monotone increasing w.r.t.
$\omega_1$ and $\omega_2$ for any $(t,x)\in\mathbb R\times\mathbb R.$

\item[(vii)]
$\bm{W}_{\omega,\omega}(t,x)=\bm{W}_{\omega,\omega}(t,-x)$
for any $\omega\in(-\infty,\varpi]$ and $(t,x)\in\mathbb R\times\mathbb R.$

\item[(viii)]
For any $(\omega_1,\omega_2)$, $(\omega_1^*,\omega_2^*)\in(-\infty,\varpi]^2$, if $\frac{\omega_1^*-\omega_1+\omega_2^*-\omega_2}{-2cT}\in\mathbb Z,$
then there exists $(t_0,x_0)\in\mathbb R\times\mathbb R$ such that
$\bm{W}_{\omega_1^*,\omega_2^*}(\cdot,\cdot)=\bm{W}_{\omega_1,\omega_2}(\cdot+t_0,\cdot+x_0)$
on $\mathbb R\times\mathbb R.$

\end{description}
\end{theorem}

\begin{proof}
For any $(\omega_1,\omega_2)\in(-\infty,\varpi]^2$ and $n\in\mathbb N^+,$
consider the following initial value problem
\begin{equation}\label{un-vn}
\begin{cases}
u_t^n=u_{xx}^n+f(t,u^n,v^n), \quad (t,x)\in (-nT,+\infty)\times \mathbb R,\\
v_t^n=dv_{xx}^n+l(t,u^n,v^n), \quad (t,x)\in (-nT,+\infty)\times \mathbb R,\\
u^n(-nT,x)=\underline{u}(-nT,x), \quad x\in \mathbb R,\\
v^n(-nT,x)=\underline{v}(-nT,x), \quad x\in \mathbb R,
\end{cases}
\end{equation}
where $\bm{\underline{w}}(t,x)=\bm{\underline{w}}(t,x;\omega_1,\omega_2)$
is defined in Lemma \ref{subsolution} for any $(t,x)\in\mathbb R\times\mathbb R.$
It then follows from \cite{lunardi1995} that problem \eqref{un-vn}
is well posed and the maximum principle holds.
For any $\bm{0}\leq\bm{w}_0(\cdot)\leq\bm{1},$
let $\bm{w}(t,x;\bm{w}_0)=(u(t,x;\bm{w}_0),v(t,x;\bm{w}_0))$
be the unique classical solution of \eqref{uv-trans-eq} defined in
$(t,x)\in[0,+\infty)\times\mathbb R$
with initial value $\bm{w}(0,x;\bm{w}_0)=\bm{w}_0(x).$
Then the unique classical solution
$\bm{w}^n(t,x)=\bm{w}^n(t,x;\omega_1,\omega_2)=(u^n(t,x),v^n(t,x))$ of \eqref{un-vn}
can be written as
$$\bm{w}^n(t,x)=\bm{w}(t+nT,x;\bm{\underline{w}}(-nT,\cdot))
\text{~~for any }(t,x)\in[-nT,+\infty)\times\mathbb R.$$
It then follows from the comparison principle that
\begin{equation*}
\bm{w}^{n+1}(-nT,x)
=\bm{w}(T,x;\bm{\underline{w}}(-(n+1)T,\cdot))
\geq\bm{\underline{w}}(-nT,x)=\bm{w}^n(-nT,x)
\text{~~for any } x\in\mathbb R.
\end{equation*}
In view of Lemmas \ref{sub-super method} and \ref{super-solu},
we see that
\begin{equation*}
\begin{cases}
\bm{\underline{w}}(t,x)\leq\bm{w}^n(t,x)\leq\bm{w}^{n+1}(t,x)
\leq\bm{1},~~t\geq-nT,~~x\in\mathbb R,\\
\bm{\underline{w}}(t,x)
\leq\bm{w}^n(t,x)
\leq\bm{\overline{W}}(t,x),~~t\in(-nT,0],~~x\in\mathbb R,
\end{cases}
\end{equation*}
where
$\bm{\overline{W}}(t,x)$
is defined in Lemma \ref{super-solu}.
By using the standard parabolic estimates and a diagonal extraction process,
there exists a subsequence $\{\bm{w}^{n_k}(t,x)\}_{k\in N}$
and a function $\bm{w}(t,x)$
such that
$\bm{w}^{n_k}(t,x),~\frac{\partial}{\partial t}\bm{w}^{n_k}(t,x)$ and
$\frac{\partial^2}{\partial {x^2}}\bm{w}^{n_k}(t,x)$ converge to
$\bm{w}(t,x),~\frac{\partial}{\partial t}\bm{w}(t,x)$ and
$\frac{\partial^2}{\partial {x^2}}\bm{w}(t,x)$ uniformly in any compact subset
$D\subset\mathbb R\times\mathbb R$ as $k\to+\infty~(n_k\to+\infty),$ respectively. Clearly, $\bm{w}(t,x)$ is well-defined in $(t,x)\in\mathbb{R}^2$ and from the equations satisfied by $\bm{w}^{n_k}(t,x),$ we conclude that
$\bm{W}_{\omega_1,\omega_2}(t,x):=\bm{w}(t,x)$
is an entire solution of \eqref{uv-trans-eq}.
Furthermore, there hold
\begin{equation}\label{uv-estimate}
\begin{cases}
\bm{\underline{w}}(t,x;\omega_1,\omega_2)\leq\bm{W}_{\omega_1,\omega_2}(t,x)\leq\bm{1},
~~x\in\mathbb R,~~t\in\mathbb R,\\
\bm{\underline{w}}(t,x;\omega_1,\omega_2)\leq\bm{W}_{\omega_1,\omega_2}(t,x)
\leq\bm{\overline{W}}(t,x),~~x\in\mathbb R,~~t\in(-\infty,0],
\end{cases}
\end{equation}
which together with Remark \ref{su-su} shows that
$\bm{W}_{\omega_1,\omega_2}(t,x)>\bm{0}$ for any $(t,x)\in\mathbb R\times\mathbb R$
and $\bm{W}_{\omega_1,\omega_2}(t,x)<\bm{1}$ for any $(t,x)\in(-\infty,0]\times\mathbb R.$
Now we prove $\bm{W}_{\omega_1,\omega_2}(t,x)<\bm{1}$
for any $(t,x)\in(0,+\infty)\times\mathbb R.$

Suppose that there exists some $(\hat{t},\hat{x})\in(0,+\infty)\times\mathbb R$
such that $U_{\omega_1,\omega_2}(\hat{t},\hat{x})=1$ or $V_{\omega_1,\omega_2}(\hat{t},\hat{x})=1.$
Let $(u(t,x),v(t,x))=(1-U_{\omega_1,\omega_2}(t,x),1-V_{\omega_1,\omega_2}(t,x)),$
then $u(\hat{t},\hat{x})=0$ or $v(\hat{t},\hat{x})=0.$
Note that $f_v,l_u\geq0,$ there hold
$$
\left[\int_0^1{f_u(t,\tau+(1-\tau)U_{\omega_1,\omega_2},
\tau+(1-\tau)V_{\omega_1,\omega_2})d\tau}\right]u
+u_{xx}-u_t\leq0
$$
and
$$
\left[\int_0^1{l_v(t,\tau+(1-\tau)U_{\omega_1,\omega_2},
\tau+(1-\tau)V_{\omega_1,\omega_2})d\tau}\right]v
+dv_{xx}-v_t\leq0,
$$
it then follows from the (strong) maximum principle that
$u(t,x)=0$ or $v(t,x)=0$
for any $(t,x)\in(-\infty,\hat{t}]\times\mathbb R,$ which contradicts to
$(u(t,x),v(t,x))>(0,0)$ for any $(t,x)\in(-\infty,0]\times\mathbb R.$
Therefore, $\bm{0}<\bm{W}_{\omega_1,\omega_2}(t,x)<\bm{1}$ for any
$(t,x)\in\mathbb R\times\mathbb R.$

\rm\textbf{(i)}
Since for any $(t,x)\in\mathbb R\times\mathbb R,$ there is
\begin{align*}
\bm{\underline{w}}(t+T,x)
&=\mathop{\max}\left\{\bm\Phi(t+T,x-ct-cT+\omega_1),\bm\Phi(t+T,-x-ct-cT+\omega_2)\right\}\\
&=\mathop{\max}\left\{\bm\Phi(t,x-ct-cT+\omega_1),\bm\Phi(t,-x-ct-cT+\omega_2)\right\}\\
&>\mathop{\max}\left\{\bm\Phi(t,x-ct+\omega_1),\bm\Phi(t,-x-ct+\omega_2)\right\}\\
&=\bm{\underline{w}}(t,x),
\end{align*}
it follows from the uniqueness of solutions and the comparison principle that
\begin{equation}\label{T-increasing}
\begin{aligned}
\bm{w}^n(t+T,x)
&=\bm{w}(t+nT+T,x;\bm{\underline{w}}(-nT,\cdot))
=\bm{w}(t+nT,x;\bm{w}(T,x;\bm{\underline{w}}(-nT,\cdot)))\\
&\geq\bm{w}(t+nT,x;\bm{\underline{w}}(T-nT,\cdot))
\geq\bm{w}(t+nT,x;\bm{\underline{w}}(-nT,\cdot))
=\bm{w}^n(t,x)
\end{aligned}
\end{equation}
for any $(t,x)\in[-nT,+\infty)\times\mathbb R.$
By taking the subsequence $\{\bm{w}^{n_k}(t,x)\}_{k\in N}$ in \eqref{T-increasing}
and letting $k\to+\infty,$ we have
$\bm{W}_{\omega_1,\omega_2}(t+T,x)\geq \bm{W}_{\omega_1,\omega_2}(t,x)$
for any $(t,x)\in\mathbb R\times\mathbb R.$
Let $\bm{W}^T_{\omega_1,\omega_2}(t,x):=\bm{W}_{\omega_1,\omega_2}(t+T,x),$
we further claim that there holds $$ \bm{W}^T_{\omega_1,\omega_2}>\bm{W}_{\omega_1,\omega_2} \text{ or } \bm{W}^T_{\omega_1,\omega_2}\equiv\bm{W}_{\omega_1,\omega_2} \text{ for all } (t,x)\in\mathbb {R}^2.$$
If the former is not true, i.e., there exists $(\tilde{t},\tilde{x})\in\mathbb R^2$ such that
$$ U^T_{\omega_1,\omega_2}(\tilde{t},\tilde{x})=U_{\omega_1,\omega_2}(\tilde{t},\tilde{x}) \text{ or }
V^T_{\omega_1,\omega_2}(\tilde{t},\tilde{x})=V_{\omega_1,\omega_2}(\tilde{t},\tilde{x}),$$
the (strong) maximum principle then implies that
$$U^T_{\omega_1,\omega_2}\equiv U_{\omega_1,\omega_2} \text{ or }
V^T_{\omega_1,\omega_2}\equiv V_{\omega_1,\omega_2} \text{ for any }
(t,x)\in(-\infty,\tilde{t}]\times\mathbb R,$$
 which together with the fact that
$f_v, ~l_u>0$ for any $(t,u,v)\in\mathbb R\times(0,1)\times(0,1)$ shows that
$(U^T_{\omega_1,\omega_2}(t,x),V^T_{\omega_1,\omega_2}(t,x))
=(U_{\omega_1,\omega_2}(t,x),V_{\omega_1,\omega_2}(t,x))$ for any
$(t,x)\in(-\infty,\tilde{t}]\times\mathbb R.$
It then follows from the comparison principle that $\bm{W}^T_{\omega_1,\omega_2}(t,x)=\bm{W}_{\omega_1,\omega_2}(t,x)$
for any $(t,x)\in\mathbb R\times\mathbb R.$
This leads to the claim.

\rm\textbf{(ii)}
For any $x\geq 0,$ by \eqref{P-} and \eqref{100}, there exists some $R_1>0$ such that
\begin{equation*}
\begin{aligned}
 0&\leq U_{\omega_1,\omega_2}(t,x)-P(t,x-ct+\omega_1)\\
  &\leq P(t,-x+j_2(t))+P(t,x+j_1(t))-P(t,x-ct+\omega_1)\\
  &\leq M_1e^{\nu_3(-x+j_2(t))}+\mathop{\sup}\limits_{(t,z)\in[0,T]\times\mathbb R}
  \left|{P_z(t,z)}\right|\cdot\left|{j_1(t)-(-ct+\omega_1)}\right|\\
  &\leq M_1e^{\nu_3j_2(t)}+R_1e^{-c\nu_3 t}
  \text{ for any } t<0,
\end{aligned}
\end{equation*}
and note that $j_2(t)\rightarrow-\infty$ as $t\rightarrow-\infty$, then
$$\mathop{\lim}\limits_{t\to-\infty}\mathop{\sup}\limits_{x\geq0}
{|U_{\omega_1,\omega_2}(t,x)-P(t,x-ct+\omega_1)|}=0.$$
The remaining parts can be treated in a similar way.

\rm\textbf{(iii)}-\rm\textbf{(v)} can be easily verified using \eqref{uv-estimate},
and \rm\textbf{(vi)} follows from
$\frac{\partial}{\partial z}\bm\Phi(t,z)>\bm{0}$ and $\bm{0}<\bm\Phi(t,z)<\bm{1}$
for any $(t,x)\in\mathbb R\times\mathbb R.$

\rm\textbf{(vii)}
Noting that
$\bm{\underline{w}}(t,x;\omega_1,\omega_2)=\bm{\underline{w}}(t,-x;\omega_1,\omega_2)$
for all $\omega_1=\omega_2=\omega\in(-\infty,\varpi]$ and $(t,x)\in\mathbb R\times\mathbb R,$
which together with the comparison principle shows that
$\bm{w}^n(t,x)=\bm{w}^n(t,-x)$ for any $n\in\mathbb N^+$ and
$(t,x)\in[-nT,+\infty)\times\mathbb R.$
It then follows that
$\bm{W}_{\omega,\omega}(t,x)=\bm{W}_{\omega,\omega}(t,-x)$
for any $\omega\in(-\infty,\varpi]$ and $(t,x)\in\mathbb R\times\mathbb R.$

\rm\textbf{(viii)}
Denote
$k^*:=\frac{\omega_1^*-\omega_1+\omega_2^*-\omega_2}{-2cT}$. If $k^*\in\mathbb Z,$ let
$t_0=k^*T$ and $x_0=\frac{\omega_1^*-\omega_1+\omega_2-\omega_2^*}{2},$
then
$$
\bm{\underline{w}}(t,x;\omega_1^*,\omega_2^*)=\bm{\underline{w}}(t+t_0,x+x_0;\omega_1,\omega_2)
\text{~~for any } (t,x)\in\mathbb R\times\mathbb R.
$$
For any $x\in\mathbb R$ and $n\in\mathbb N^+$ with $n^{\prime}:=n-k^*>0,$
we further have
$$\bm{\underline{w}}(-nT,x;\omega_1^*,\omega_2^*)
=\bm{\underline{w}}(-nT+t_0,x+x_0;\omega_1,\omega_2)
=\bm{\underline{w}}(-n^{\prime}T,x+x_0;\omega_1,\omega_2),
$$
and hence $\bm{w}^n(t,x;\omega_1^*,\omega_2^*)=\bm{w}^{n^{\prime}}(t+t_0,x+x_0;\omega_1,\omega_2)$
for any $(t,x)\in[-nT,+\infty)\times\mathbb R.$
Note that $n\to+\infty$ if and only if $n^{\prime}\to+\infty,$
then we get that
$\bm{W}_{\omega_1^*,\omega_2^*}(\cdot,\cdot)=\bm{W}_{\omega_1,\omega_2}(\cdot+t_0,\cdot+x_0)$
on $\mathbb R\times\mathbb R.$ The proof is complete.
\end{proof}

In order to study the convergence of the entire solution
$\bm{W}_{\omega_1,\omega_2}(t,x)$ as $\omega_1,\omega_2\to-\infty,$
we first introduce a pair of sub- and supersolutions constructed in \cite[Lemma 3.4]{bao2013}.

Let
$\zeta$ be a smooth function satisfying
$\zeta(x)=0$ for $x\leq-2,$ $\zeta(x)=1$ for $x\geq2,$
$0\leq\zeta^\prime(x)\leq1$ and $\left|\zeta^{\prime\prime}(x)\right|\leq1.$
Define $\bm{p}(x,t)=(p_1(x,t),p_2(x,t))$ as
\begin{equation*}
\begin{aligned}
p_1(x,t)=\zeta(x)\phi_1(t)+(1-\zeta(x))\phi_0(t),\\
p_2(x,t)=\zeta(x)\psi_1(t)+(1-\zeta(x))\psi_0(t),
\end{aligned}
\end{equation*}
where $(\phi_i(t),\psi_i(t))~(i=0,1)$ satisfy
\begin{equation}\label{phi-psi-0}
\begin{cases}
\phi_0^\prime(t)=f_u(t,0,0)\phi_0(t)+\lambda_0\phi_0(t),\\
\psi_0^\prime(t)=l_u(t,0,0)\phi_0(t)+l_v(t,0,0)\psi_0(t)+\lambda_0\psi_0(t),\\
\phi_0(t+T)=\phi_0(t),~\psi_0(t+T)=\psi_0(t)
\end{cases}
\end{equation}
and
\begin{equation}\label{phi-psi-1}
\begin{cases}
\phi_1^\prime(t)=f_u(t,1,1)\phi_1(t)+f_v(t,1,1)\psi_1(t)+\lambda_1\phi_1(t),\\
\psi_1^\prime(t)=l_v(t,1,1)\psi_1(t)+\lambda_1\psi_1(t),\\
\phi_1(t+T)=\phi_1(t),~\psi_1(t+T)=\psi_1(t),
\end{cases}
\end{equation}
respectively. Note that, if we take
$\lambda_0=-\overline{f_u(t,0,0)}>0$ and $\lambda_1=-\overline{l_v(t,1,1)}>0,$
then (A3) ensures that
$(\phi_i(t),\psi_i(t))~(i=0,1)$ exist and are strictly positive and T-periodic functions (see \cite[section 3]{bao2013}).
According to Bao and Wang \cite[Lemma 3.4]{bao2013},
there exist some positive constants $\beta_1,~\sigma_0$ and $\delta_0$ such that
for any $\delta\in(0,\delta_0),~\sigma_1\geq\sigma_0$ and $\xi^\pm\in\mathbb R$
the functions $\bm{w}^\pm(t,x)=(u^\pm(t,x),v^\pm(t,x))$ defined on $t>0$ by
\begin{equation*}
\bm{w}^{\pm}(t,x)=\bm\Phi(t,x-ct+\xi^{\pm}\pm\sigma_1\delta(1-e^{-\beta_1t}))\pm
\delta\bm{p}(x-ct+\xi^{\pm}\pm\sigma_1\delta(1-e^{-\beta_1t}),t)e^{-\beta_1t}
\end{equation*}
are super- and subsolutions of the following auxiliary
cooperative system, respectively:
\begin{equation}\label{auxi-sys}
\begin{cases}
u_t=u_{xx}+F(t,u,v),~~t\in\mathbb R^+,~x\in\mathbb R,\\
v_t=dv_{xx}+L(t,u,v),~~t\in\mathbb R^+,~x\in\mathbb R,\\
(u(x,0),v(x,0))=(u_0(x),v_0(x))\in\mathcal C,~x\in\mathbb R,
\end{cases}
\end{equation}
where
$\mathcal C:=\{\bm{w}\in C(\mathbb R,\mathbb R\times\mathbb R)
:-\bm{1}\leq\bm{w}(x)\leq\bm{2}\},$ and
$$F(t,u,v)=f(t,u,v)+L_1\{u\}^-v,~L(t,u,v)=l(t,u,v)+L_2\{1-v\}^-(u-1)$$
with
$L_1=\mathop{\max}\limits_{t\in[0,T]}\{b_1(t)q(t)\},
~L_2=\mathop{\max}\limits_{t\in[0,T]}\{a_2(t)p(t)\}$
and $\{w\}^-:=\max\left\{-w,0\right\}.$

\begin{theorem}\label{convegence}
For any $(\omega_1,\omega_2)\in(-\infty,\varpi]^2,$ the following convergence hold
in the sense of locally in $(t,x)\in\mathbb R\times\mathbb R:$
\begin{equation*}
\bm{W}_{\omega_1,\omega_2}(t,x)
\text{ converges to }
\left\{{\begin{array}{*{20}{c}}
\begin{aligned}
&\bm\Phi(t,x-ct+\omega_1) ~\text{ as } \omega_2\to-\infty,\\
&\bm\Phi(t,-x-ct+\omega_2) ~\text{ as } \omega_1\to-\infty, \\
&\bm{0} ~\text{ as } \omega_1\to-\infty \text{ and }\omega_2\to-\infty.
\end{aligned}
\end{array}} \right.
\end{equation*}
\end{theorem}

\begin{proof}
Only the assertion that
$\bm{W}_{\omega_1,\omega_2}(t,x)$ converges to
$\bm\Phi(t,x-ct+\omega_1)$ locally in $(t,x)\in\mathbb R\times\mathbb R$
as $\omega_2\to-\infty$ will be proved here,
since the others can be discussed similarly.
Taking a sequence $\left\{(\omega_1,\omega_2^k)\right\}_{k\in\mathbb Z}$ with
$(\omega_1,\omega_2^k)\in(-\infty,\varpi]^2$ such that
$\omega_2^{k+1}<\omega_2^k<\omega_1$ for any $k\in\mathbb Z$ and
$\omega_2^k\to-\infty$ as $k\to+\infty.$
According to Theorem \ref{existence},
there exist entire solutions
$\bm{W}_{\omega_1,\omega_2^k}(t,x)$ of \eqref{uv-trans-eq}
such that for any $(t,x)\in(-\infty,0]\times\mathbb R$ and $k\in\mathbb Z,$
\begin{align*}
\bm\Phi(t,x-ct+\omega_1)
&\leq\max\left\{\bm\Phi(t,x-ct+\omega_1),\bm\Phi(t,-x-ct+\omega_2^{k+1})\right\}\\
&\leq\bm{W}_{\omega_1,\omega_2^{k+1}}(t,x)
\leq\bm{W}_{\omega_1,\omega_2^k}(t,x)\\
&\leq\bm\Phi(t,x+\tilde{j_1}(t;\omega_1))+\bm\Phi(t,-x+\tilde{j_2}(t;\omega_1,\omega_2^k)).
\end{align*}
Then there exists a function $\bm{W}(t,x)=(U(t,x),V(t,x))$ such that
$\bm{W}_{\omega_1,\omega_2^k}(t,x)$ converges to $\bm{W}(t,x)$
locally in $(t,x)\in\mathbb R\times\mathbb R$ as $k\to+\infty.$
Moreover,
\begin{equation*}
\bm\Phi(t,x-ct+\omega_1)\leq\bm{W}(t,x)\leq\bm\Phi(t,x+\tilde{j_1}(t;\omega_1))
\text{ for any } (t,x)\in(-\infty,0]\times\mathbb R.
\end{equation*}
Fix any $t_1<0,$ let
\begin{equation*}
\eta:=\max\left\{\sup\limits_{x\in\mathbb R}
\left|U(t_1,x)-P(t_1,x-ct_1+\omega_1)\right|,
\sup\limits_{x\in\mathbb R}
\left|V(t_1,x)-Q(t_1,x-ct_1+\omega_1)\right|\right\},
\end{equation*}
and denote
\begin{align*}
&p_*=\min\left\{\mathop{\inf}\limits_{x\in\mathbb R,t\in(0,T)}p_1(x,t),
\mathop{\inf}\limits_{x\in\mathbb R,t\in(0,T)}p_2(x,t)\right\}>0,\\
&p^*=\max\left\{\mathop{\sup}\limits_{x\in\mathbb R,t\in(0,T)}p_1(x,t),
\mathop{\sup}\limits_{x\in\mathbb R,t\in(0,T)}p_2(x,t)\right\}>0.
\end{align*}
Recall that $\tilde{j_1}(t;\omega_1)-(-ct+\omega_1)\to0$ as $t\to-\infty,$
then there is some $t_2<t_1$ such that
for any $\delta\in(0,\delta_0],$ we have
\begin{equation*}
\bm\Phi(t_2,x-ct_2+\omega_1)
\leq\bm{W}(t_2,x)
\leq\bm\Phi(t_2,x-ct_2+\omega_1)+\delta p_* \quad\text{for any } x\in\mathbb R.
\end{equation*}
By comparison principle,
we see that for any $(t,x)\in[t_2,+\infty)\times\mathbb R,$
\begin{align*}
&\bm\Phi(t,x-ct+\omega_1)\\
&\leq\bm{W}(t,x)\\
&\leq\bm\Phi(t,x-ct+\omega_1+\xi(t-t_2))+
\delta\bm{p}(x-ct+\omega_1+\xi(t-t_2),t)e^{-\beta_1(t-t_2)},
\end{align*}
where $\xi(t)=\sigma_1\delta(1-e^{-\beta_1t})$
with $\sigma_1$ large enough.
Then it follows that for any $x\in\mathbb R,$
\begin{equation*}
\begin{aligned}
&\bm\Phi(t_1,x-ct_1+\omega_1)\\
&\leq\bm{W}(t_1,x)\\
&\leq\bm\Phi(t_1,x-ct_1+\omega_1+\xi(t_1-t_2))+
\delta\bm{p}(x-ct_1+\omega_1+\xi(t_1-t_2),t)e^{-\beta_1(t_1-t_2)}\\
&\leq\bm\Phi(t_1,x-ct_1+\omega_1+\sigma_1\delta)+\delta p^*,
\end{aligned}
\end{equation*}
which further implies that
\begin{equation*}
0\leq\bm{W}(t_1,x)-\bm\Phi(t_1,x-ct_1+\omega_1)
\leq\delta\left(\sigma_1\cdot\mathop{\sup}\limits_{(t,z)\in[0,T]\times\mathbb R}
  \left|{P_z(t,z)}\right|+p^*\right)
\text{ for any } x\in\mathbb R.
\end{equation*}
Noting that $\delta\in(0,\delta_0]$ is arbitrary, it then follows that
$\eta=0,$ which further implies that $\bm{W}_{\omega_1,\omega_2}(t,x)$ converges to $\bm\Phi(t,x-ct+\omega_1)$ locally in $(t,x)\in\mathbb R\times\mathbb R$ as $\omega_2\to-\infty.$
\end{proof}

\begin{proof}[\textbf{Proof of Theorem \ref{ES}}]
For any given contants $\theta_1,\theta_2\in\mathbb R$, there exists $n^*\in \mathbb{Z}$ such that
$$\omega_1:=\theta_1+cn^*T\in(-\infty,\varpi] \text{ and }
\omega_2:=\theta_2+cn^*T\in(-\infty,\varpi].$$
Then by Theorem \ref{existence}, system \eqref{uv-trans-eq} admits an entire solution $\bm{W}_{\omega_1,\omega_2}(t,x)$ defined on $\mathbb{R}^2$.
Let
$
\bm{W}_{\theta_1,\theta_2}(\cdot,\cdot)
=(U_{\theta_1,\theta_2}(\cdot,\cdot),V_{\theta_1,\theta_2}(\cdot,\cdot))
:=\bm{W}_{\omega_1,\omega_2}(\cdot+n^*T,\cdot).
$
Moreover, we have
\begin{align*}
\bm{W}_{\theta_1,\theta_2}(t,x)&=\bm{W}_{\omega_1,\omega_2}(t+n^*T,x)\\
&\geq\bm{\underline{w}}(t+n^*T,x;\omega_1,\omega_2)\\
&=\mathop{\max}\left\{\bm\Phi(t+n^*T,x-c(t+n^*T)+\omega_1),\bm\Phi(t+n^*T,-x-c(t+n^*T)+\omega_2)\right\}\\
&=\mathop{\max}\left\{\bm\Phi(t,x-ct+\theta_1),\bm\Phi(t,-x-ct+\theta_2)\right\},
\end{align*}
which implies the last convergence of \rm\textbf{(vi)} in Theorem \ref{ES}.
Other statements in Theorem \ref{ES} can be easily verified by virtue of Theorems
\ref{existence} and \ref{convegence}.
\end{proof}

\begin{proof}[\textbf{Proof of Theorem \ref{ES-c>0}}]
Consider the case $c>0.$
Assume that $\bm\Psi(t,z)=\bm\Psi(t,x-ct)$
is an increasing traveling wave solution
of \eqref{uv-trans-eq} satisfying
$\bm\Psi(t,-\infty)=\bm{0} \text{ and } \bm\Psi(t,+\infty)=\bm{1}.$
Let $\tilde{c}=-c<0$ and define
$$\bm{\widetilde{\Psi}}(t,z)=\bm{\widetilde{\Psi}}(t,x-\tilde{c}t)
=\bm{1}-\bm\Psi(t,-(x+ct)),$$
then
$\bm{\widetilde{\Psi}}(t,-\infty)=\bm{0},$
$\bm{\widetilde{\Psi}}(t,+\infty)=\bm{1}$ and
$\frac{\partial}{\partial z}\bm{\widetilde{\Psi}}(t,z)>\bm{0}.$
It is not difficult to verify that
$\bm{\widetilde{\Psi}}(t,z)=(\widetilde{P}(t,z),\widetilde{Q}(t,z))$ is a
traveling wave solution of the following system
\begin{equation}\label{uv-c>0}
\begin{cases}
\tilde{u}_t=\tilde{u}_{xx}+(1-\tilde{u})[b_1q\tilde{v}-a_1p\tilde{u}],\\
\tilde{v}_t=d\tilde{v}_{xx}+\tilde{v}[b_2q(1-\tilde{v})-a_2p(1-\tilde{u})],
\end{cases}
\end{equation}
and we know from (C2) that
$\frac{\widetilde{Q}(t,z)}{\widetilde{P}(t,z)}\geq\eta_1$
for any $(t,z)\in\mathbb R\times(-\infty,0].$
Similar to the argument for system \eqref{uv-trans-eq},
we know that for any $\theta_1,\theta_2\in\mathbb R,$
system \eqref{uv-c>0} admits an entire solution
$\bm{0}<\bm{W}(t,x)<\bm{1}$ satisfying
$\bm{W}(t+T,x)=\bm{W}(t,x)$ or $\bm{W}(t+T,x)>\bm{W}(t,x)$
for any $(t,x)\in\mathbb R\times\mathbb R,$
and
\begin{equation*}
\begin{aligned}
\lim_{t\to-\infty}
&\left\{\sup_{x\geq 0}
{\left|\bm{W}(t,x)-\bm{\widetilde\Psi}(t,x-\tilde{c}t-\theta_1)\right|}\right .\\
&+\left .\sup_{x\leq 0}
{\left|\bm{W}(t,x)-\bm{\widetilde\Psi}(t,-x-\tilde{c}t-\theta_2)\right|}
\right\}
=0.
\end{aligned}
\end{equation*}
Define
$\bm{\widetilde{W}}_{\theta_1,\theta_2}(t,x)=\bm{1}-\bm{W}(t,x),$
then $\bm{\widetilde{W}}_{\theta_1,\theta_2}(t,x)$
is an entire solution of system \eqref{uv-trans-eq} which satisfies
\begin{align*}
\lim_{t\to-\infty}
&\left\{\sup_{x\geq 0}
{\left|\bm{\widetilde{W}}_{\theta_1,\theta_2}(t,x)-\bm\Psi(t,-x-ct+\theta_1)\right|}\right .\\
&+\left .\sup_{x\leq 0}
{\left|\bm{\widetilde{W}}_{\theta_1,\theta_2}(t,x)-\bm\Psi(t,x-ct+\theta_2)\right|}
\right\}
=0.
\end{align*}
Furthermore, we can see that the other assertions in Theorem \ref{ES-c>0} are valid.
\end{proof}

\section*{Acknowledgments}

W.T. Li was partially supported by NSF of China (11671180) and FRFCU (lzujbky-2016-ct12).
J.B. Wang was partially supported by FRFCU (lzujbky-2016-226).

\end{document}